\documentclass[12pt]{amsart}
\title[Quasisymmetries of Sierpi\'nski carpet Julia sets]{Quasisymmetries of 
Sierpi\'nski carpet\\ Julia sets}
\author{Mario Bonk,  Misha Lyubich, and Sergei Merenkov}

\address{Mario Bonk\\ Department of Mathematics\\
University of California, Los Angeles\\ Los Angeles,
CA 90095\\USA} \email{mbonk@math.ucla.edu}
\thanks{M.B.\ was supported by NSF grants DMS-0456940, DMS-0652915, DMS-1058283, DMS-1058772, and DMS-1162471}

\address{Mikhail Lyubich\\Institute for Mathematical Sciences\\
Stony Brook University\\
Stony Brook, NY 11794\\USA }
\email{mlyubich@math.sunysb.edu}
\thanks{M.L.\ was supported by NSF grant 
DMS-1007266.}

\address{Sergei Merenkov\\Department of Mathematics\\University of Illinois at Urbana-Champaign\\Urbana, IL 61801\\USA, \textnormal{and}
Department of Mathematics\\The City College of New York,  New York, NY 10031, USA.
}

\email{merenkov@illinois.edu, smerenkov@ccny.cuny.edu}
\thanks{S.M.\ was supported by NSF grant DMS-1001144.}

\date{March 3, 2014}

\usepackage[matrix,arrow,curve,cmtip]{xy}
\usepackage{amsmath, amssymb, amsthm,latexsym}
\usepackage{graphicx}

\newcommand\C{{\mathbb C}}
\newcommand\hC{\widehat {\mathbb C}}
\newcommand\N{{\mathbb N}}

\newcommand\D{{\mathbb D}}

\newcommand\Z{{\mathbb Z}}
\newcommand\R{{\mathbb R}}

\newcommand\Sph{{\mathbb S}}

\newcommand\Ju{\mathcal J}
\newcommand\Fa{\mathcal F}
\newcommand\dee{\partial}

\newcommand\dist{\operatorname{dist}}
\newcommand\diam{\operatorname{diam}}

\newcommand\crit{\operatorname{crit}}
\newcommand\bran{\operatorname{crit}}
\newcommand\id{\operatorname{id}}
\newcommand\inte{\operatorname{int}}

\renewcommand\:{\colon}
\newcommand\sub {\subseteq}
\newcommand\ra {\rightarrow}

\newcommand\Om{\Omega}
\newcommand\ga{\gamma}

\newcommand\eps{\epsilon}

\newcommand\no{\noindent} 
\newcommand\post{\operatorname{post}} 

\numberwithin{equation}{section}

\newtheorem{theorem}{Theorem}[section]

\newtheorem{corollary}[theorem]{Corollary}

\newtheorem{lemma}[theorem]{Lemma}

\theoremstyle{definition}

\def\note#1{\marginpar{1}}

\begin{document}

\abstract{We prove that if $\xi$ is a quasisymmetric homeomorphism between Sierpi\'nski carpets that are the Julia sets of post\-criti\-cally-finite rational maps, then $\xi$ is the restriction  of a M\"obius transformation to the  Julia set. This implies that the group of quasisymmetric homeomorphisms of a Sierpi\'nski carpet Julia set of a postcritically-finite rational map is finite. 
}\endabstract

\maketitle

\section{Introduction}\label{s:Intro}
\no 
A {\em Sierpi\'nski carpet} is a topological space homeomorphic to the well-known standard Sierpi\'nski carpet fractal. A subset $S$ of  the 
Riemann sphere $\hC$ is a Sierpi\'nski carpet if and only if it has empty interior and  can be written as 
\begin{equation}\label{eq:Scarp}
 S=\hC\setminus \bigcup_{k\in \N} D_k,
 \end{equation} where 
the sets $D_k\sub \hC$, $k\in \N$, are pairwise disjoint  Jordan regions with $\partial D_k\cap \partial D_l=\emptyset$ for $k\ne l$, and $\diam(D_k)\ra 0$ as $k\to \infty$. A Jordan curve $C$ in a Sierpi\'nski carpet $S$ is called a {\em peripheral circle} if its removal does not separate $S$. If $S$ is a Sierpi\'nski carpet as in \eqref{eq:Scarp}, then the peripheral circles of $S$ are precisely the boundaries $\partial D_k$ of the Jordan regions $D_k$. 

 Sierpi\'nski carpets can arise as Julia sets 
of rational maps.  For example,  in~\cite{BDLSS} (see also~\cite{DL}) it is shown that for the family 
$$
f_\lambda(z)=z^2+\lambda/z^2,
$$ 
in each neighborhood of $0$ in the parameter plane, 
there are infinitely many parameters $\lambda_n$  such that
 the Julia sets  $\Ju(f_{\lambda_n})$  are  Sierpi\'nski carpets on which 
 the corresponding maps $f_{\lambda_n} $ are not topologically conjugate.  

\subsection{Statement of main results}\label{SS:MR}
Sierpi\'nski carpets  have large groups of self-homeomorphisms.  For example, it follows from results in \cite{gW58} that if  $n\in\N$ and we are   given  two $n$-tuples of distinct peripheral circles of a Sierpi\'nski carpet $S$, then there exists a homeomorphism on $S$ that takes the peripheral circles of the first $n$-tuple to the corresponding peripheral circles of the other.

In contrast, strong rigidity statements are valid  if we consider quasisymmetries of Sierpi\'nski carpets.  By a \emph{quasisymmetry} of a compact set $K$ we mean a quasisymmetric homeomorphism of $K$ onto itself (see Section~\ref{ss:qc}).  

In the present paper,   we prove  quasisymmetric rigidity results for 
Sierpi\'nski carpets that are Julia sets
of  postcritically-finite rational maps (see Section~\ref{ss:compldyn} for the definitions). 

\begin{theorem} \label{thm:main1} Let $f \: \hC\ra \hC$ be a postcritically-finite rational  map, and suppose the Julia set $\Ju(f)$ of $f$ is a Sierpi\'nski carpet. If  $\xi$ is a quasisymmetry of $\Ju(f)$, then $\xi$ is the restriction to $\Ju(f)$ of a M\"obius transformation on $\hC$.  
\end{theorem}
Here a {\em M\"obius transformation} is a fractional linear transformation on $\hC$ or a complex-conjugate of such a map; so a M\"obius transformation can be orientation-reversing.


An example of a rational map as in the statement of Theorem~\ref{thm:main1} is $f(z)=z^2-1/(16 z^2)$; see Figure~\ref{F:JSSC} for its Julia set. 
\begin{figure}
[htbp]
\begin{center}
\includegraphics[height=50mm]{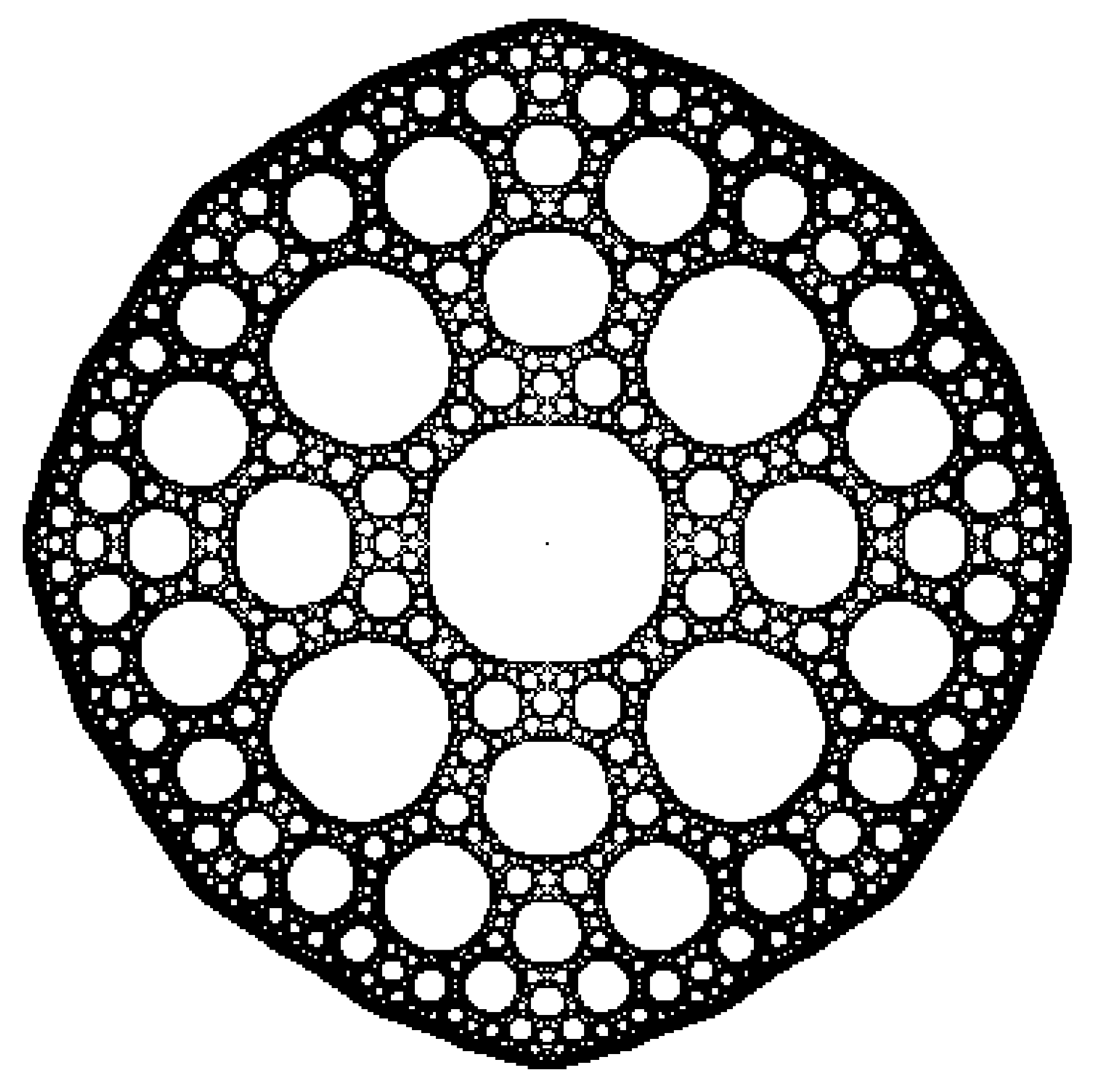}
\caption{
The Julia set of $f(z)=z^2-\frac1{16 z^2}$.
}
\label{F:JSSC}
\end{center}
\end{figure}

The six critical points of this map are $0, \infty$, and $\sqrt[4]{-1}/2$. The postcritical set consists of four points: $0, \infty$, and $\sqrt{-1}/2$. The point  $\infty$ is a fixed point and forms the only periodic cycle; so $f$ is a hyperbolic postcritically-finite rational map. The group $G$ of M\"obius transformations that leave $\Ju(f)$ invariant contains the maps $\xi_1(z)=i z,\ \xi_2(z)=\overline z$, and $\xi_3(z)=1/(4z)$. It is likely that these maps actually generate $G$.  

It was shown by Levin~\cite{Le, Le1} (see also~\cite{LP})
that if $f$ is a hyperbolic rational map that is not equivalent to $z^d,\ d\in\Z$, and whose Julia set $\Ju(f)$ is neither the whole sphere, a circle,  nor an arc of a circle, then the group of M\"obius transformations that keep $\Ju(f)$ invariant is finite. 

In our context, we have the following   corollary  to Theorem~\ref{thm:main1}.

\begin{corollary}\label{C:FinGp}
If $\Ju(f)$ is a Sierpi\'nski carpet Julia set of a postcritical\-ly-finite rational map $f$, then the group of quasisymmetries  of $\Ju(f)$ is finite. 
\end{corollary}

See the end of Section~\ref{s:proof} for the proof. An immediate consequence is the following 
fact; see~\cite[Corollary~1.3]{Me4} for a related result.

\begin{corollary}\label{C:RatGroup}
No Sierpi\'nski carpet Julia set of a postcritically-finite rational map is quasisymmetrically equivalent to the limit set of a Klei\-nian group.
\end{corollary}
Indeed, the Kleinian  group acts on its limit set, and so  limit set has an infinite group of 
quasisymmetries. 

Theorem~\ref{thm:main1} is a special case of the following statement, which is the  main result
of this paper.

\begin{theorem}\label{thm:main2}
Let $f , g\: \hC\ra \hC$ be postcritically-finite  rational  maps, and suppose the corresponding Julia sets $\Ju(f)$ and $\Ju(g)$  are Sierpi\'nski carpets. If  $\xi$ is a quasisymmetric homeomorphism of $\Ju(f)$ onto $\Ju(g)$, then $\xi$ is the restriction to $\Ju(f)$ of a M\"obius transformation on $\hC$. 
\end{theorem}

Here and in Theorem~\ref{thm:main1}   the map  $\xi$ does not respect the dynamics  {\it a priori}; in particular, we do not assume that $\xi$ conjugates $f$ and $g$.  If $\xi$ conjugates the two rational maps, then the result is  
well known.  It follows  from the uniqueness part of Thurston's characterization of postcritically-finite   maps that can be extracted from \cite{DH}.

The statement of Theorem~\ref{thm:main2} is false if we drop  the assumption that the maps are postcritically-finite.  
Indeed, each hyperbolic rational map $f_0$ is quasiconformally structurally stable near its Julia set, 
i.e., for any rational map $f$  sufficiently close to $f_0$, there exist 
(backward invariant) neighborhoods $U_0\supset \Ju(f_0)$ and 
$U\supset \Ju(f)$ and a quasiconformal homeomorphism $h: U_0\ra U$ such that $h(f_0 (z)) = f(h(z))$ for each $z\in f_0^{-1}(U_0)$
(see \cite{MSS}). This conjugacy is quasisymmetric on the Julia set. Moreover, $f$ can be selected so that $h$ is not a M\"obius transformation
(essentially due to the fact that a rational map of degree $d\geq 2$ depends on $2d+1>3 $ complex parameters, 
while a M\"obius transformation depends only on  $3$).
This discussion is applicable, e.g., to the previously mentioned hyperbolic rational map  $f_0(z)=z^2-1/(16 z^2)$ 
whose   Julia set is Sierpi\'nski carpet. 



\subsection{Previous rigidity results for quasisymmetries}\label{SS:PRR} 
In~\cite{BKM}  the authors considered quasisymmetric rigidity questions for {\em Schottky sets}, i.e., subsets of the $n$-dimensional sphere $\Sph^n$ whose complements are collections of open balls with disjoint closures. A Schottky set $S$ is said to be {\em rigid} if each quasisymmetric map of $S$ onto any other Schottky set is the restriction to $S$ of a M\"obius transformation. The main result  in dimension $n=2$ is the following theorem.

\begin{theorem}\label{T:BKMT1.2}\cite[Theorem~1.2]{BKM}
A Schottky set in $\Sph^2$ is rigid if and only if it has spherical measure zero.
\end{theorem}

In higher dimensions the ``if" part of this statement is still true, but not the ``only if" part.

A {\em relative Schottky set} in a region $D\subseteq\Sph^n$ is a subset of $D$ obtained by removing from $D$ a collection of balls whose closures are contained in $D$ and are pairwise disjoint. 

\begin{theorem}\label{T:BKM8.1}\cite[Theorem~8.1]{BKM}
A quasisymmetric map from a locally porous relative Schottky set $S$ in $D\subseteq\Sph^n$ for $n\geq 3$ onto a relative Schottky set $\tilde S\subseteq\Sph^n$ is the restriction to $S$ of a M\"obius transformation. 
\end{theorem}
For the definition of a {\em locally porous} relative Schottky set see the 
end of Section~\ref{s:Schottky} (where we define it only for  $n=2$, but the definition is essentially the same for $n\ge 3$). 

Note that M\"obius transformations preserve the classes of Schottky and relative Schottky sets. 

Theorem~\ref{T:BKM8.1} is no longer true in dimension $n=2$. However, the third author proved the following result.

\begin{theorem}\label{T:Me2T1.2}\cite[Theorem~1.2]{Me2}
Suppose that $S$ is a relative Schottky set of measure zero in a Jordan region $D\subset\C$. Let $f\: S\to\tilde S$ be a locally quasisymmetric orientation-preserving homeomorphism from $S$ onto a relative Schottky set $\tilde S$ in a Jordan region $\tilde D\subset\C$. Then $f$ is conformal in $S$ in the sense that
$$
f'(p) =\lim_{q\in S,\, q\to p}\frac{f(q)-f(p)}{q-p}
$$
exists  for every $p\in S$ and is not equal to zero. Moreover, the map $f$ is locally bi-Lipschitz and the derivative $f'$ is continuous  on $S$.
\end{theorem}

We will call conformal  maps $f$ as above {\em Schottky maps}, see Section~\ref{s:Schottky}.   

Another fractal space for which a quasisymmetric rigidity result has been established is the {\em slit carpet}  $S_2$ obtained as follows. We start with a closed unit square $[0,1]^2$ in the plane and subdivide it into $2\times2$ equal subsquares in the obvious way. We then create a vertical slit from the middle of the common vertical side  of the top two subsquares to the middle  of the common vertical side of the bottom  two  subsquares. Finally, we repeat these procedures for all the subsquares  and continue indefinitely. The metric on $S_2$ is the path metric induced from the plane.  

\begin{theorem}\label{T:Me1T6.1}\cite[Theorem~6.1]{Me1}
The group of quasisymmetries of $S_2$ is the group of isometries of $S_2$, that is, the finite dihedral group $D_2$ consisting of four elements  and isomorphic to $\Z/2\Z\times \Z/2\Z$.
\end{theorem}

Finally, the first and the third authors of this paper proved the following rigidity result for the standard Sierpi\'nski carpet.

\begin{theorem}\cite[Theorem 1.1]{BM}\label{T:BMT1.1}
Every quasisymmetry of the standard Sierpi\'nski carpet $S_3$ is a Euclidean isometry.
\end{theorem} 
For a  class of standard square carpets   a slightly weaker result is also known, namely that the group of quasisymmetries of such a carpet is   finite dihedral  \cite[Theorem~1.2]{BM}.

\subsection{Main techniques  and an outline of the proof of Theorem~\ref{thm:main2}}\label{SS:MT}
The methods used to prove the main result of this paper are  different from those employed to establish Theorems~\ref{T:BKMT1.2}--\ref{T:BMT1.1}.

The first key ingredient in the proof of Theorem~\ref{thm:main2}  is well known and  basic in complex dynamics, namely  the use of {\em conformal elevators}. Roughly speaking, this means that by using the dynamics of a given subhyperbolic rational map one can ``blow up"  a small disk centered in the Julia set to a definite size. We will need a  careful 
 analysis  to control analytic and geometric properties of such blow-ups (see Section~\ref{s:confelev}).  

We will apply this to  establish uniform  geometric properties of the peripheral circles of 
Sierpi\'nski carpets that arise as Julia sets of subhyperbolic rational 
maps. 

\begin{theorem}\label{thm:circgeom}
Let $f\: \hC\ra \hC$ be a subhyperbolic rational map whose Julia set $\Ju(f)$ is a Sierpi\'nski carpet. Then the peripheral circles of 
the Sierpi\'nski carpet $\Ju(f)$ are uniform quasicircles, they are uniformly relatively separated, and they occur on all locations and scales. 
Moreover, $\Ju(f)$ is a porous set, and in particular, has measure zero.
\end{theorem}

See  Sections~\ref{ss:qc} and \ref{s:geomper} for an explanation of  the terminology and for  the proof. 

The  fact that the Julia set of  subhyperbolic rational map has  measure zero is well known \cite{Ly}; moreover,  for  hyperbolic rational maps  the previous  theorem is fairly standard.

 A quasisymmetric  map $\xi$ on a Julia set as in the previous theorem can be extended to 
  a quasiconformal map on $\hC$ by  the following fact. 

\begin{theorem}\label{T:BoP5.1}
Suppose that $\{D_k:k\in \N\}$ is a  family of   Jordan regions in $\hC$
with pairwise
disjoint closures, and let 
$$\xi\: S = \hC\setminus\bigcup_{k\in \N} D_k \to \hC$$ be a
quasisymmetric embedding. If the Jordan curves $\dee D_k$, $k\in \N$, form a family 
of   uniform quasicircles,
then there exists a quasiconformal homeomorphism 
 $F\: \hC \to \hC$ whose restriction to $S$ is equal to $f$.
\end{theorem}
This follows from \cite[Proposition~5.1]{Bo}, where a stronger, quantitative version is formulated.  

Now suppose  that, as in Theorem~\ref{thm:main2},   
  we have  a given quasisymmetry $\xi$ of the Julia set $\Ju(f)$ of a postcritically-finite rational map $f$ onto the Julia set of another such map $g$. Since every postcritically-finite map is subhyperbolic, we can apply  Theorems~\ref{thm:circgeom} and \ref{T:BoP5.1} and extend $\xi$ (non-uniquely) to a quasiconformal map on $\hC$,  also denoted by $\xi$. We then use conformal elevators to produce a sequence $\{h_k\}$ of uniformly quasiregular maps defined in a fixed disk centered at a suitably chosen point in  $\Ju(f)$. The map $h_k$ is a local symmetry of $\Ju(f)$ and  is given by 
$$
h_k=\xi^{-1}\circ g^{m_k}\circ\xi\circ f^{-n_k}
$$
for appropriate choices of $n_k,m_k\in\N$ and  branches of $f^{-n_k}$.
One can also write $h_k$ in the form
$h_k=g_\xi^{m_k}\circ f^{-n_k}$, where $g_\xi=\xi^{-1}\circ g\circ\xi$. This has the advantage that the maps involved are defined   on  $\Ju(f)$. 
Standard compactness arguments imply that the sequence $\{h_k\}$ has a subsequence that converges locally uniformly to a non-constant quasiconformal map $h$ on a disk centered at 
a point in $\Ju(f)$.  

 Now one  wants to show that the sequence $\{h_k\}$ stabilizes and so $h_k=h=h_{k+1}$ for large $k$. From this one can  derive an equation of the form 
\begin{equation}\label{eq:funddynrel}
 g^{m'}\circ \xi = g^m\circ \xi \circ f^n 
\end{equation}
on $\Ju(f)$ for some integers $m,m',n\in \N$. This  relates $\xi$ to the dynamics of $f$ and $g$.

When $\xi$ is a M\"obius transformation, this approach 
 goes back to Levin~\cite{Le} (see also \cite{Le1} and \cite{LP}).    
 In this case, the maps involved are analytic, which played a crucial role  
  in establishing that the sequence 
 $\{h_k\}$ stabilizes. 

In our case, the maps are not analytic, so we need to invoke a different techniques, 
namely  rigidity results for  Schottky maps on relative Schottky sets (see Section~\ref{s:Schottky}) established by  the third author.  
Our situation can be reduced to the Schottky setting 
 due  to  the following quasisymmetric uniformization result:

\begin{theorem}\label{T:BoC1.2}\cite[Corollary~1.2]{Bo}
Suppose that $T\subseteq\hC$ is a Sier\-pi\'n\-ski carpet whose peripheral circles
are uniform quasicircles that are uniformly relatively separated . Then $T$ can be mapped to a
round Sier\-pi\'n\-ski carpet $S\sub \hC$
by a quasisymmetric homeomorphism $\beta\: T\to S$.
\end{theorem}

Here we say that a Sierpi\'nski carpet $S$ in $\hC$ is {\em round} if all of its peripheral circles are geometric circles; in this case $S$ is a Schottky set. According to Theorem~\ref{T:BoP5.1}, we can again extend the map $\beta$  to a quasiconformal map on the sphere.


If we apply the previous theorem to $T=\Ju(f)$, then 
it follows from Theorem~\ref{T:Me2T1.2} that each map $h_k$ is conjugate by $\beta$ to a Schottky map  defined in a fixed relatively open  $S\cap V\sub S$. This is a  conformal map on $S\cap V$ in the sense of  Theorem~\ref{T:Me2T1.2}, with a continuous derivative. 
After conjugation by $\beta$ we can write this map  in the form $h_k=g_\xi^{m_k}\circ f^{-n_k}$ as well, 
where, by abuse of notation, we do not distinguish between the original maps and their conjugates by $\beta$. The following result implies that the sequence $\{h_k\}$ of these conjugate maps stabilizes.  

\begin{theorem}\label{T:Me3T5.3}\cite[Theorem~5.2]{Me3}
Let $S$ be a locally porous relative Schottky set in a region $D\subseteq\C$,  let $p\in S$, and let $U$ be an open neighborhood of $p$ such that $S\cap U$ is connected. Suppose that there exists a Schottky map $u\: S\cap U\to S$ with $u(p)=p$ and $u'(p)\neq 1$. 

Let $\{h_k\}_{k\in\N}$ be a sequence of Schottky maps $h_k\: S\cap U\to S$. 
If each map $h_k$ 
is a homeomorphism onto its image and if the sequence $\{h_k\}$ converges locally uniformly to a homeomorphism $h$, then there exists $N \in \N$ such that $h_k =h$ in $S\cap U$ for all $k\geq N$.
\end{theorem}

 After some algebraic manipulations, the relation  $h_k=h_{k+1}$ for large enough $k$
 gives the  equation 
 \begin{equation}\label{eq:dynrel2}
 g_\xi^{m'} = g_\xi^m\circ f^n 
\end{equation}
in a relatively  open set $S\cap U$,
for some integers $m,m',n\in \N$.  
The following result  promotes  the validity of this equation to the whole set $S$.

 
 \begin{theorem}\label{C:Me3C4.2}\cite[Corollary~4.2]{Me3}
 Let $S$ be a locally porous relative Schottky set in $D \subseteq \C$, and let $U \subset D$ be an open set such that $S \cap U$ is connected. Let $u,v\: S \cap U \to S$ be Schottky maps, and 
$E = \{p \in S \cap U \: u(p) = v(p)\}$. If $E$ has an accumulation point in $U$, then $E = S\cap U$ and so $u=v$.  
 \end{theorem}

If we conjugate  back by $\beta^{-1}$ to the original maps,  then we can conclude that  \eqref{eq:dynrel2},
and hence \eqref{eq:funddynrel}, is  valid on the whole Julia set $\Ju(f)$. 

In the final part of the argument, we  analyze  functional equation  \eqref{eq:funddynrel}
on  $\Ju(f)$.
 Lemma~\ref{L:Rot} in combination with  auxiliary results established in 
 Section~\ref{ss:compldyn} allow us to conclude that $\xi$ has a conformal extension into each periodic component of the Fatou set of $f$. Using the  description of the dynamics on the Fatou set in the postcritically-finite case (Lemma~\ref{lem:FatouDyn}), we can actually produce  a conformal extension of $\xi$ into every Fatou component of $f$ by a lifting procedure (see Lemma~\ref{lem:lifting}). These extensions piece  together  a global quasiconformal map, again denoted by $\xi$,  that is conformal on the Fatou set of $f$. Since the Julia set $\Ju(f)$ has measure zero, it follows that $\xi$ is $1$-quasiconformal and hence a M\"obius transformation
 (see Lemma~\ref{lem:ext}).

\medskip\noindent 
\textbf{Acknowledgments.} The authors would like to thank A.~Eremenko and  A.~Hinkkanen for providing useful references.

\section{Quasiconformal maps and related concepts}\label{ss:qc}
\no
Throughout the paper we assume that the reader is familiar with basic notions and facts from the theory of quasiconformal maps and complex dynamics. We will review some relevant definitions and statements related to these topics in this and the following sections.

We will almost exclusively deal with the  complex plane $\C$ equipped with the Euclidean metric (the distance between $z$ and $w$ denoted by $|z-w|$), or the Riemann sphere $\hC$ 
equipped with the chordal metric $\sigma$. We denote by $\D=\{z\in \D: |z|<1\} $ the open unit  disk in $\C$. 
By default a set $M\sub \hC$ carries (the restriction of) the chordal metric $\sigma$,  but we will usually specify the relevant metric in a given context.

 With an  underlying space $X$ and a metric on $X$ understood, we denote by $B(p,r)$  the open ball of radius $r>0$ centered at $p\in X$, by 
 $\diam(M)$ the diameter, and by $\dist(M,N)$ the distance of sets $M,N\sub X$.
 The cardinality of set $M$ is $\#M\in \N_0\cup\{\infty\}$, and $\text{id}_M$ the identity map on $M$. If $f\:X\ra Y$ is a map between sets 
 $X$ and $Y$ and $A\sub X$, then we denote by $f|_A\:A\ra Y$  the restriction of $f$ to $A$.  
 
  We will now discuss quasiconformal and related maps. For general background on this topic we refer to \cite{Ri, AIM, Va, He}. 
  
 A non-constant  continuous map $f \: U \to \C$
on a  region $U\sub \C$ is called \emph{quasiregular} if $f$ is   in the Sobolev space $W_{\rm loc}^{1,2}$
and if there exists a constant $K\ge 1$ such that the (formal) Jacobi   matrix $Df$
satisfies
\begin{equation}\label{eq:defqr}
||Df(z)||^2\leq K\det(Df(z))
\end{equation}
for almost every $z \in U$.  The condition that $f\in W_{\rm loc}^{1,2}$ means that the first distributional partial derivatives of $f$ are locally in the Lebesgue space $L^2$. 
This definition requires only local coordinates and hence the notion of a quasiregularity 
 can be extended to maps $f\: U\ra \hC$ on regions $U\sub \hC$. If $f$ is a homeomorphism onto its image in addition, then $f$ is called  a {\em quasiconformal} map.  The map $f$ is called {\em locally quasiconformal} if for every point $p\in U$ there exists a region $V$ with $p\in V\sub U$  such that $f|_V$ is quasiconformal.

Each quasiregular map $f\: U\to \hC$ is a {\em branched covering map}. This means that $f$  is an open map and  the preimage $f^{-1}(q)$ of each point $q\in \hC$ is 
a discrete subset of $U$. A point $p\in U$, near which the quasiregular map  $f$ is not a local homeomorphism, is called a {\em critical point} of $f$. These points are isolated in $U$, and accordingly, the set 
 $\crit(f)$ of all critical points of $f$ is a discrete and relatively closed subset of $U$.
  One can easily derive these statements from the fact that every quasiregular map $f\: U\ra \hC$ on a region $U\sub \hC$   can  be represented locally
 in the form 
$f=g\circ \varphi$, where $g$ is a holomorphic map and $\varphi$ is quasiconformal
\cite[p.~180, Corollary 5.5.3]{AIM}.



Let $(X, d_X)$ and $(Y,d_Y)$ be metric spaces and 
let $f\: X\to Y$ be a homeomorphism. 
The map $f$ is called {\em quasisymmetric} if there exists a homeomorphism $\eta\:
[0,\infty)\to[0,\infty)$ such that 
\begin{equation}\label{eq:defqs}
\frac{d_Y(f(u),f(v))}{d_Y(f(u),f(w))}\leq\eta
\bigg(\frac{d_X(u,v)}{d_X(u,w)}\bigg),
\end{equation}
for every triple of distinct points $u,v,w\in X$.

Suppose  $U$ and $V$ are  subregions of $\hC$. Then 
every orientation-preserving quasisymmetric homeomorphism $f\: U\ra V$ is quasiconformal. Conversely, every quasiconformal homeomorphism $f\: U\ra V$ is {\em locally quasisymmetric}, i.e., 
for every compact set $M\sub U$,  
the restriction $f|_M\: M\ra f(M)$ is a quasisymmetry (see \cite[p.~58, Theorem 3.4.1 and p.~71, Theorem~3.6.2]{AIM}).  

Often it is important to keep track of the quantitative information  that appears in the definition of quasiregular maps as in \eqref{eq:defqr} or  quasisymmetric  maps as in \eqref{eq:defqs}.
Then we speak of a $K$-quasiregular map, or an $\eta$-quasisymmetry, etc.

A Jordan curve $J\sub \hC$ is called a {\em quasicircle}  if there exists a quasisymmetry 
$f\: \partial \D\ra J$. This is equivalent to the requirement that there exists a constant 
$L\ge 1$ such that 
\begin{equation}\label{eq:qarccond}
 \diam(\alpha)\le L\sigma(u,v), 
 \end{equation} whenever 
$u,v\in \N$, $u\ne v$,  and $\alpha$ is the smaller subarc of $J$ with endpoints $u,v$. 

  If $\{J_k: k\in \N\}$  is a family of quasicircles , then we say that it consists of {\em uniform quasicircles} if condition 
\eqref{eq:qarccond} is true for some constant  $L\ge 1$ independent of $k\in \N$.  
The family $\{J_k: k\in \N\}$ is said to be {\em uniformly relatively separated} if there exists 
a constant $c>0$ such that 
$$ \frac{\dist(J_k, J_l)}{\min\{ \diam(J_k), \diam(J_l)\}}\ge c$$ 
for all $k,l\in \N$, $k\ne l$.

We conclude this section with  an extension result that is needed  in the proof of  
Theorem~\ref{thm:main2}. 
  
 \begin{lemma}\label{lem:ext} Let $S\sub \hC$  be a Sierpi\'nski carpet written in the form 
 $S =\hC\setminus\bigcup_{k\in \N}D_k$ with pairwise disjoint  Jordan regions $D_k\sub \hC$, 
 and suppose that the peripheral circles $\partial D_k$, $k\in \N$,  of $S$ are uniform quasicircles.
Let $\xi\: S\ra \hC$ be an orientation-preserving quasisymmetric embedding of $S$ and suppose that each restriction 
$\xi|_{\partial D_k}\: \partial D_k\ra \hC$, $k\in \N$, extends to an embedding  $\xi_k\:\overline D_k\ra \hC$ that is conformal on $D_k$. Then $\xi$ has a unique quasiconformal extension 
$\widetilde \xi\:\hC\ra \hC$   that is conformal on $\hC\setminus S$. 

Moreover, if $S$ has measure zero, then   $\widetilde \xi$ is a M\"obius transformation.  
\end{lemma}

Here we say that the embedding  $\xi\: S\ra \hC$ is  {\em orientation-preserving} if $\xi$ has an extension
to an orientation-preserving homeomorphism on the whole sphere $\hC$. This does not depend on the embedding and  is equivalent to the following statement: if we orient each peripheral circle $\partial_k D$ of the Sierpi\'nski carpet as in the theorem so that $S$ lies ``to the left" of $\partial D_k$ and if we equip $\xi(\partial D_k)$ with the induced orientation, then $\xi(S)$ lies to the left of $\xi(\partial D_k)$. 
  
\begin{proof} The proof relies on the  rather subtle, but well-known relation between 
quasiconformal, quasisymmetric, and quasi-M\"obius maps (for the definition of the latter class and related facts see \cite[Section~3]{Bo}).
In the proof we will omit some details that can easily be extracted from the considerations  in \cite[Section~5]{Bo}.  

 Under the given assumption  the image $S'=\xi(S)$ is also a
Sierpi\'n\-ski carpet that we can represent in the form 
 $S'=\hC\setminus\bigcup_{k\in \N}D'_k$ with pairwise disjoint  Jordan regions $D'_k$.
 Since $\xi$ maps the peripheral circles of $S$ to the peripheral circles of $S'$, we can choose the labeling so that $\xi(\partial D_k)=\partial D_k'$ for $k\in \N$.
 
For each embedding  $\xi_k\: \overline D_k \ra \hC$ as in the statement of the lemma, 
 we necessarily have $\xi_k(\overline D_k)=\overline D'_k$, because $\xi_k(\partial D_k)=\xi(\partial D_k)=\partial D'_k$ and  $\xi$ is 
orientation-preserving. 
 Moreover, $\xi_k$ is uniquely determined by $\xi|_{\partial D_k}$; this follows from the classical fact that 
a homeomorphic extension of a conformal map between given Jordan regions is uniquely determined by the image of three distinct  boundary points. 

Our original map $\xi$ and the unique maps $\xi_k$, $k\in \N$, piece together to a  homeomorphism $\tilde \xi \: \hC \ra \hC$   that is conformal  on $\hC\setminus S$. Moreover, $\tilde \xi$ is the unique homeomorphic extension of $\xi$ with this conformality property.

The Jordan curves $\partial D_k$, $k\in \N$, form a family of uniform quasicircles, and hence also their images $\partial D'_k=\xi(\partial D_k)$, $k\in \N$, under the quasisymmetry $\xi$. 
This implies that Jordan regions $D_k$ and $D'_k$   are {\em uniform quasidisks}. More precisely,  
there exist $K$-quasiconformal homeomorphisms $\alpha_k$ and $\alpha'_k$ on $\hC$  such that $\alpha_k(\overline D_k)=\overline \D$ and $\alpha'_k(\overline D'_k)=\overline \D$ for $k\in \N$, where $K\ge 1$ is independent of $k$. Then the maps $\alpha_k$ and $\alpha'_k$
are uniformly quasi-M\"obius (see \cite[Proposition~3.1~(i)]{Bo}).
Moreover, the homeomorphisms $\alpha'_k\circ \xi_k\circ \alpha_k^{-1}\: \overline \D \ra \overline \D$ are uniformly quasiconformal on $\D$, and hence also uniformly quasi-M\"obius
on $\overline \D$ (the last implication  is essentially well-known; one can reduce to  \cite[Proposition~3.1~(i)]{Bo} by Schwarz reflection in $\partial \D$ and use the fact that $\partial \D$ is removable for quasiconformal maps \cite[Section 35]{Va}). It follows that the maps
$\xi_k$, $k\in \N$,  are uniformly quasi-M\"obius.

Now  the maps $\xi_k|_{\partial D_k}=\xi|_{\partial D_k}$, $k\in \N$, are actually uniformly quasisymmetric, because $\xi$ is a quasisymmetric embedding. This implies that  the family   $\xi_k$, $k\in \N$, is also uniformly quasisymmetric (this can be seen as in the proof of \cite[Proposition~5.3~(ii)]{Bo}). Since $\xi=\tilde \xi|_S$ is quasisymmetric 
 and the maps $\xi_k=\tilde \xi|_{\overline D'_k}$ for  $k\in \N$ are  uniformly 
 quasisymmetric,  the homeomorphism $\tilde \xi$ is quasiconformal (this is shown as in the last part of the proof of \cite[Proposition~5.1]{Bo}). 

If $S$ has measure zero, then $\tilde \xi$ is a quasiconformal map that is conformal on the set $\hC\setminus S$ of full measure in $\hC$.
Hence $\tilde \xi$ is $1$-quasiconformal on $\hC$, which, as is well-known, implies that $\tilde \xi$ is a 
M\"obius transformation (one can, for example, derive this from the uniqueness part of Stoilow's factorization theorem \cite[p.~179, Theorem 5.5.1]{AIM}). 
\end{proof}

\section{Fatou components of postcritically-finite maps} \label{ss:compldyn} \no
In this section we record some facts related to complex dynamics. For  basic definitions and general background we refer to standard sources such as  \cite{Be, CG, Mi2, St}.
 
Let   $f$ be a rational map on the Riemann sphere $\hC$ of degree $\deg(f)\ge 2$. We denote 
by  $f^n$ for 
$n\in \N$ the $n$th-iterate of $f$, by  $\Ju(f)$ its Julia set  and by  $\Fa(f)$ its  Fatou set.
Then $$\Ju(f)=f(\Ju(f))=f^{-1}(\Ju(f))=\Ju(f^n), $$ and we have similar relations for the Fatou set. We will use these standard facts throughout.

A continuous map $f\: U\ra V$  between two regions $U,V\subseteq\hC$ is called {\em proper}   if for every compact set $K\sub V$ the set 
$f^{-1}(K)\sub  U$ is also compact.  If $f$ is a rational map on $\hC$, then its  restriction
$f|_{U}$ to $U$ is a proper  map $f|_{U}\: U\ra V$ if and only if $f(U)\sub V$ and $f(\partial U)\sub \partial V$. 

If $f$ is a rational map, it is a proper map between its Fatou components; more precisely, if $U$ is a Fatou component of $f$, then $V=f(U)$ is also a Fatou component of $f$, and the restriction is a proper map of $U$ onto $V$. In particular, the boundary of each Fatou component is mapped onto the boundary of another Fatou component.
Similarly, we can always  write 
$$f^{-1}(U)=U_1\cup \dots \cup  U_m,$$ 
where $U_1, \dots, U_m$ are Fatou components of $f$ that are mapped properly onto $U$.  


If $f\:U\ra V$ is a proper holomorphic map, possibly defined on a larger set than $U$, then the {\em topological degree} $\deg(f, U)\in \N$ of $f$ on $U$ is well-defined as the unique number such that 
$$ \deg(f, U)=\sum_{p\in f^{-1}(q)\cap U}\deg_f(p)$$
for all $q\in V$, where  $\deg_f(p)$ is the local degree of $f$ at $p$. 

Suppose that  $U\sub \hC$ is {\em finitely-connected}, i.e.,  $\hC\setminus U$ has only finitely 
many connected components, and let $k\in \N_0$ be the number  of components of $\hC\setminus U$. We call 
$\chi(U)=2-k$ the {\em Euler characteristic} of $U$ (see  \cite[Section~5.3]{Be} for a related discussion). 
The quantity $\chi(U)$ is invariant under homeomorphisms and can be obtained as a limit  of Euler characteristics of polygons (defined in the usual way as  for simplicial complexes)  forming a suitable  exhaustion of $U$. 
We have $\chi(U)=2$ if and only if $U=\hC$. So $\chi(U)\le 1$ for 
finitely-connected proper subregions $U$ of $\hC$ with    $\chi(U)=1$
if and only if $U$ is simply connected;

If $U$ and $V$ are finitely connected regions, and $f\: U\ra V$ is a proper holomorphic map, then a version of the Riemann-Hurwitz relation (see \cite[Section~5.4]{Be} and \cite[Chapter~1, Section~6]{St}) says that 
\begin{equation}\label{eq:RHform}
 \deg(f, U)\chi(V)=\chi(U)+\sum_{p\in U} (\deg_f(p)-1). 
 \end{equation}
Part of this statement is that the sum on the right-hand side of this identity is defined as it has only  finitely many non-vanishing terms. 

The Riemann-Hurwitz formula is valid in a limiting sense for regions  that are {\em infinitely connected}, i.e.,  not finitely-connected. In this case, the  relation simply says that if $U,V\sub \hC$ are regions and $f\:U\ra V$ is a proper holomorphic map, then $U$ is infinitely connected if and only if $V$ is infinitely connected.

If $f$ is a rational map on $\hC$, then a  point is called a {\em postcritical point} of $f$ if it is the image of a critical point of $f$ under some iterate of $f$. If we denote the set of these points by $\post(f)$, then we have 
$$ \post(f)=\bigcup_{n\in \N}f^n(\crit(f)). $$ 
The map $f$ is called {\em postcritically-finite} if every critical point has a finite orbit under iteration of $f$. This  is equivalent to the requirement that $\post(f)$ is a finite set. 
Note  that $\post(f)=\post(f^n)$ for all $n\in \N$. 

We denote by  $\post^{c}(f)\sub \post(f)$ the set of  points that lie in cycles of periodic critical points; so 
$$\post^c(f)=\{ f^n(c): n\in \N_0 \text{ and $c$ is a periodic critical point of $f$} \}. $$

 If $f$ is postcritically-finite, then $f$ can only have one possible type of periodic Fatou components (for the general classification of periodic Fatou components and their relation to critical points see \cite[Sections~7.1 and 9.1]{Be}); namely,  every periodic Fatou component $U$ is a {\em B\"ottcher domain} for some iterate of $f$: 
there exists an iterate $f^n$ and a superattracting fixed point $p$ of $f^n$ such that $p\in U$.   

The following, essentially well-known, lemma describes the dynamics of a 
postcritically-finite rational map on a fixed  Fatou component. 
Here and in the following we will use the notation $P_k$ for the 
{\em $k$-th power map} given by $P_k(z)=z^k$ for $z\in \C$, where $k\in \N$.

\begin{lemma}[Dynamics  on fixed Fatou components]\label{lem:postratFatou} 
Suppose $f$ is a postcritically-finite rational map, and $U$ a  Fatou component
of $f$ with $f(U)=U$. Then $U$ is simply connected, and contains precisely 
one critical point $p$ of $f$. We have $f(p)=p$ and $ U\cap\post(f)=\{p\}$, and there exists 
a conformal map $\psi\: U\ra \D$  with $\psi(p)=0$ such that $\psi\circ f\circ \psi^{-1}=P_k$, where $k=\deg_f(p)\ge 2$.  
\end{lemma}

So $p$ is a superattracting fixed point of $f$,  $U$ is the corresponding B\"ottcher domain of $p$, and on $U$ the map $f$ is conjugate to a power map.  Note that in general the map $\psi$ is not uniquely determined due to a rotational ambiguity; namely, one can replace $\psi$ with $a\psi$, where $a^{k-1}=1$. 

\begin{proof} As the statement is essentially well-known, we will only give a sketch 
of the proof. 

 By the classification of Fatou components it is clear  
that $U$ contains a superattracting fixed point $p$. Then $f(p)=p$ and $p$ is a critical point of $f$. Let $k=\deg_f(p)\ge 2$. Without loss of generality we may assume that $p=0$, and $\infty\not\in U$. 
Then there exists a holomorphic  function  $\varphi$ (the B\"ottcher function) defined 
in a neighborhood of $0$ with $\varphi(0)=0=p$, $\varphi'(0)\ne 0$, 
and 
\begin{equation}\label{eq:Boett}
f(\varphi(z))=\varphi(z^k)
\end{equation}
 for $z$ near $0$ \cite[Chapter 3, Section 3]{St}.  

Since the maps $f^n$, $n\in \N$,  form  a normal family on $U$ and $f^n(z)\to p$ for $z$ near $p$, we have  
\begin{equation}\label{eq:superattraction}
\text{$f^n(z)\to p=0$ as $n\to \infty$ locally uniformly for $z\in U$.}
\end{equation}
 
Let $r\in (0,1]$ be the maximal radius such that $\varphi$ has a  holomorphic extension  to  the Euclidean disk $B=B(0,r)$. Then \eqref{eq:Boett} remains valid on $B$. We claim that $r=1$, and so $B=\D$;
otherwise, $0<r<1$, and by using  \eqref{eq:Boett} and the fact that $f\: U\ra U$ is proper, 
 one can show that $\overline {\varphi (B)}\sub U$. The equation  \eqref{eq:Boett} implies 
that every point $q\in \varphi(B)\setminus \{p\}$ has an infinite orbit under iteration of $f$; by the local uniformity of the convergence in \eqref{eq:superattraction} this remains true for $q\in  \overline{\varphi(B)}\setminus \{p\}$. 
Since $f$ is postcritically-finite, this implies that  no point in $\overline{\varphi(B)}\setminus \{p\}$ can be a critical point of $f$; but then   \eqref{eq:Boett} allows us to  holomorphically extend 
$\varphi$ to a disk $B(0,r')$ with $r'>r$. This is a contradiction showing that indeed $r=1$ and $B=\D$.

As before by using   \eqref{eq:Boett}, one sees that  $\varphi(\D)\sub U$. Actually, one  also observes that for  points $q=\varphi(z)$ with $z\in \D$ closer and closer  to $\partial \D$, the convergence  $f^n(q)\to p$ is at a slower and slower rate. By \eqref{eq:superattraction} this is only possible if $\varphi(z)$ is close to $\partial U$ if $z\in \D$ is close to $\partial \D$; in other words, $\varphi$ is a proper map of $\D$ to $U$ and in particular $\varphi(\D)=U$. 

It follows from \eqref{eq:Boett} that $\varphi$ cannot have any critical points in $\D$ (to see this, argue by contradiction and consider a critical point $c\in \D$ of $\varphi$ with smallest absolute value $|c|$). The Riemann-Hurwitz formula \eqref{eq:RHform} then implies that $\chi(U)=1$ and $\deg(\varphi)=1$. In particular,  $U$ is simply connected and $\varphi$ is a conformal map of  $\D$ onto $U$. For the conformal map $\psi=\varphi^{-1}$ from $U$ onto $\D$ we then have 
$\psi(p)=0$ and we get 
the desired relation $\psi\circ f\circ \psi^{-1}=P_k$. This relation (or again  \eqref{eq:Boett})
 implies that the fixed point $p$ is the only critical point of $f$ in $U$ and that  each point $q\in U\setminus \{p\}$ has an  infinite orbit; so $p$ is  the only postcritical point of $f$ in $U$. 
 \end{proof}

The following lemma gives us control for the mapping behavior of iterates of a rational map
onto regions containing at most one postcritical point.

\begin{lemma}\label{lem:deg}
Let $f\: \hC\ra \hC$ be a  rational map, $n\in \N$, $U\sub 
\hC$ be  a simply connected region with $\#\hC\setminus U\ge 2$ and
$\#(U\cap \post(f))\le 1$, and $V$ be a component of $f^{-n}(U)$.
 Let $p\in U$ be the unique point in $U\cap \post(f)$ if  $\#(U\cap \post(f))=1$ and $p\in U$ be arbitrary if   $U\cap \post(f)=\emptyset$, and let $\psi_U\: U\ra \D $ be a conformal map with $\psi_U(p)=0$.

Then $V$ is simply connected,  the map $f^n\: V\ra U$ is proper, and  
there 
exists $k\in \N$,  and a conformal map $\psi_V\:V\ra \D$
with  $\psi_U\circ f^n=P_k\circ \psi_V$.

Here $k=1$ if $U\cap \post(f)=\emptyset$. Moreover, if  $U\cap \post^c(f)=\emptyset$ then $k\le N$, where $N=N(f)\in \N$ is a constant only depending on $f$.  \end{lemma}
 
In particular, for given $f$ the number $k$ is uniformly bounded by a constant $N$ independently of $n$ and $U$ if $U\cap \post^c(f)=\emptyset$.
 
 \begin{proof}  Under the given assumptions,  $V$ is a region, and the map
 $g:=f^n|_V\: V\ra U$ is proper.
 
 Since  $U\cap\post(f)\sub \{p\}$, the point $p$ is the only possible critical value of $g$. It  follows from the Riemann-Hurwitz formula  \eqref{eq:RHform} that 
\begin{eqnarray*}
\chi(V) &=&\deg(g,V)\chi(U)-\sum_{z\in V}(\deg_{g}(z)-1)\\ 
&=& \deg(g,V)-(\deg(g,V)-\#g^{-1}(p))\, =\, \#g^{-1}(p).
\end{eqnarray*}
As $\chi(V)\le 1$, this is only possible if $\chi(V)=1$ and $\#g^{-1}(p)=1$;  so $V$ is simply connected and $p$ has precisely one preimage $q$ in $V$ which is the only possible critical point of $g$. Obviously, $\#\hC\setminus V\ge 2$, and so 
 there exists a  conformal map  $\psi_V\:  V\ra \D$ with $\psi_V(q)=0$. Then 
$(\psi_U\circ f^n\circ \psi_V^{-1})$ is a proper holomorphic map from $\D$ to itself  and hence a finite Blaschke product $B$.
Moreover, $B^{-1}(0)=\{0\}$, and so  we can replace $\psi_V$ by
a postcomposition with  a suitable rotation around $0$ so that $B(z)=z^k$ for $z\in \D$, where $k=\deg(g)\in \N$. If $U\cap \post(f)=\emptyset$, then $q$ cannot be  a critical point of $g$, and so 
$k=1$. 

It remains to produce a uniform upper bound for $k$ if we assume in addition 
 that $U\cap \post^c(f)=\emptyset$. Then in the list $q, f(q), \dots, f^{n-1}(q)$ each critical point of $f$ can appear at most once; indeed, otherwise the list contains a periodic critical point which implies that    $p=f^n(q)\in  U\cap \post^c(f)$,  contradicting  our additional hypothesis. 
  
We conclude  that 
$$k=\deg_{f^n}(q)=\prod_{i=0}^{n-1} \deg_f(f^i(q))\le 
 N=N(f):=\prod_{c\in \crit(f)}\deg_f(c),$$
which gives the desired uniform upper bound for $k$. 
\end{proof}

The next lemma  describes the dynamics of a  postcritically-finite rational map on arbitrary Fatou components. 

\begin{lemma}[Dynamics on the Fatou components]\label{lem:FatouDyn} Let $f\: \hC\ra \hC$ be a  postcritically-finite rational map, and $\mathcal{C}$ be the collection of all Fatou components of $f$. Then there exists a family $\{\psi_U\:U\ra \D: U\in \mathcal{C}\}$ of conformal maps with the following property: 
if $U$ and $V$ are Fatou components of $f$ with $f(V)=U$, then 
\begin{equation}\label{eq:desFatcomp}
\psi_U\circ f=P_k\circ \psi_V
\end{equation}
 on $V$ for some $k=k(U,V)\in \N$. 

Moreover, for each $U\in \mathcal{C}$ the point $p_U:=\psi_U^{-1}(0)$ is  the unique point in $U\cap\bigcup_{n\in \N_0} f^{-n}(\post(f))$. 
\end{lemma}

In contrast to the points $p_U$  the maps $\psi_U$ are not uniquely determined in general due to a certain   rotational freedom. As we will see in the proof of the lemma, $p_U$ can also be characterized as the unique point in $U$ with a finite orbit under iteration of $f$.  In the following, we will
choose   $p_U$  as a basepoint in the Fatou component 
$U$. If we take $0$ as a basepoint in $\D$, then in the previous lemma 
we get the following commutative diagram of basepoint-preserving maps between 
{\em pointed} 
 regions (i.e., regions with a distinguished basepoint):
\begin{equation*}
  \xymatrix {
 (V,p_V)   \ar[r]^{\psi_V}  \ar[d]_{f}  & (\D,0)   \ar[d]^{P_k}   \\
    (U,p_U)  \ar[r]^{\psi_U}  &   (\D,0)
  }
\end{equation*}
Note that this implies in particular that $f^{-1}(p_U)=\{p_V\}$ and that $f\: V\setminus \{p_V\}\ra 
U\setminus\{p_U\}$ is a covering map. 

\begin{proof} We first construct  the desired maps $\psi_U$ for the  periodic Fatou components
$U$ of $f$. So fix  a periodic Fatou component $U$ of $f$, and let $n\in \N$ be the period 
of $U$, i.e., if we  define $U_0:=U$ and $U_{k+1}=f(U_k)$ for $k=0, \dots, n-1$, then 
the Fatou components $U_0, \dots, U_{n-1}$ are all distinct, and $U_n=f^n(U)=U$. 
 
By Lemma~\ref{lem:postratFatou} applied to the map $f^n$, for each $k=0, \dots, n-1$  the Fatou component $U_k$ is simply connected and there exists a unique point $p_k\in U_k$ that lies in $\post(f)=\post(f^n)$. Moreover, there 
exists  a conformal map $\psi_0\: U_0\ra \D$ with $\psi_0(p_0)=0$ such that 
$\psi_{0}\circ f^n=P_{d}\circ \psi_{0}$ for suitable $d\in \N$. 

Let $\psi_1\: U_1\ra \D$ be a conformal map with $\psi_1(p_1)=0$. By the argument in the proof of Lemma~\ref{lem:deg} we know that $B=\psi_1\circ f \circ \psi_0^{-1}$ is a finite Blaschke product $B$ with $B^{-1}(0)=\{0\}$ and so $B(z)=az^{d_1}$ for  suitable constants $d_1\in\N$ and $a\in \C$ with $|a|=1$. By adjusting $\psi_1$ by a suitable rotation factor if necessary, we may assume that $a=1$. Then $\psi_1\circ f=P_{d_1}\circ \psi_0$ on $U_0$. If  we repeat this argument, then we get conformal maps $\psi_k\:U_k\ra \D$ with $\psi_k(p_k)=0$ and 
\begin{equation}\label{eq:perdFatok}
\psi_{k}\circ f=P_{d_k}\circ \psi_{k-1}
\end{equation} on $U_{k-1}$ with suitable $d_k\in \N$ for 
$k=1, \dots, n$. Note that 
$$ \psi_n\circ f^n = P_{d_n}\circ \psi_{n-1}\circ f^{n-1}=\dots=P_{d_n}\circ \dots\circ P_{d_1}\circ \psi_0=P_{d'}\circ \psi_0$$
on $U_0$, where $d'\in \N$. On the other hand, $\psi_0\circ f^n=P_d\circ \psi_0$ by definition of $\psi_0$.
Hence $d=\deg_{f^n}(p_0)=d'$, and so $ \psi_n\circ f^n=  \psi_0\circ f^n$ on $U_0$ which implies $\psi_n=\psi_0$. If we  now define  $\psi_{U_k}:=\psi_k$ for $k=0, \dots, n-1$, then by \eqref{eq:perdFatok} the desired relation \eqref{eq:desFatcomp} holds for each suitable pair 
of Fatou components from the  cycle $U_0, \dots, U_{n-1}$.  We also choose $p_{U_k}=p_k\in 
U_k$ as a basepoint in $U_k$ for $k=0, \dots, n-1$. We know that $p_k$ is the unique point in 
$U_k$ that lies in $\post(f)$.  Since $f$ is postcritically-finite, each point in 
 $P:=\bigcup_{n\in \N_0} f^{-n}(\post(f))$ has a finite orbit under iteration of $f$.
It follows from Lemma~\ref{lem:postratFatou} that 
each  point $p\in U_k\setminus \{p\}$ has an infinite orbit  and therefore cannot lie in 
$P$.  Hence $p_{U_k}$ is the unique point in $U_k$ that lies in $P$. 

We repeat this argument for the other finitely many periodic Fatou components $U$
to obtain suitable conformal maps $\psi_U\: U\ra \D$ and unique  basepoints $p_U=\psi_U^{-1}(0)\in U\cap P$.

If $V$ is a non-periodic Fatou component, then it is mapped to a periodic Fatou component 
by a sufficiently high iterate of $f$ (this is Sullivan's theorem on the non-existence of wandering domains; see \cite[p.~176, Theorem~8.1.2]{Be}). We call the smallest number $k\in \N_0$ such that $f^k(V)$ is a periodic Fatou component the {\em level} of $V$.

Suppose $V$ is an arbitrary Fatou component of  level $1$. Then $U=f(V)$ is periodic, and so $\psi_U$ and $p_U$ are already defined and we know that $\{p_U\}=\post(f)\cap U$. Hence by Lemma~\ref{lem:deg} there exists a conformal map $\psi_V\: V\ra U$ such that \eqref{eq:desFatcomp} is valid.
If $p_V:=\psi_V^{-1}(0)$, then $f(p_V)=p_U\in \post(f)$, and so $p_V\in V\cap P$.
Moreover, \eqref{eq:desFatcomp} shows that $f(V\setminus \{p_V\})=U\setminus\{p_U\}$ which implies that each point in $V\setminus \{p_V\}$ has an infinite orbit and cannot lie in $P$. It follows that $p_V$ is the unique point in $V$ that lies in $P$. 

We repeat this argument for Fatou components of higher and higher level. Note that 
if for a Fatou component $U$  a conformal map $\psi_U\: U\ra \D$ has already been constructed and we know that  $p_U:=\psi^{-1}(0)$ is the unique point in $U\cap P$, then
$U\cap \post(f) \sub \{p_U\}$ and we can again apply Lemma~\ref{lem:deg} for a  Fatou component  $V$ with $f(V)=U$. 

In this way we obtain conformal maps $\psi_U$ as desired for all Fatou components $U$.
 The point $p_U=\psi_U^{-1}(0)$ is the unique point in $U$ that lies in $P$, because $f^k(p_U)\in \post(f)$ for some $k\in\N_0$ and all other points in $U$ have an infinite orbit.  
\end{proof}

We conclude this section with a lemma that is required in the proof of Theorem~\ref{thm:main2}.

\begin{lemma}[Lifting lemma]\label{lem:lifting} Let $f\: \hC\ra \hC$ be a  postcritically-finite rational map, $n\in \N$,
and  $(U,p_U)$ and   $(V,p_V)$ be pointed Fatou components of $f$ that are Jordan regions with $f^n(U)=V$. Suppose $D\sub \hC$ is another Jordan region with a basepoint $p_D\in D$, and suppose that $\alpha\: \overline D\ra \overline V $ is a map with the following properties:

\begin{itemize}
\item[\textnormal{(i)}] $\alpha$ is continuous on $ \overline D$ and holomorphic  on
 $D $, 

\smallskip 
\item[\textnormal{(ii)}]  $\alpha^{-1}(p_V)=\{p_D\}$, 

\smallskip 
\item[\textnormal{(iii)}] there exists a continuous map $\beta\: \partial D \ra \partial U$ with 
$f^n\circ \beta = \alpha|_{\partial D}$. 
\end{itemize} 
Then there exists a unique continuous map $\tilde \alpha\: \overline D\ra\overline U $ 
with $f^n\circ \tilde \alpha= \alpha$ and $\tilde \alpha|_{\partial D}=\beta$. Moreover, $\tilde \alpha $ is 
holomorphic on $D$ and satisfies $\tilde \alpha^{-1}(p_U)=\{p_D\}$. 

If, in addition,  $\beta$ is a homeomorphism of $\partial D$ onto $\partial U$, then $\tilde \alpha$ is a conformal homeomorphism of $\overline D$ onto $ \overline U$.
\end{lemma}

Here we call a map $\varphi\: \overline \Om \ra \overline \Om'$ between the closures of two Jordan regions $\Om, \Om'\sub \hC$ a {\em conformal homeomorphism}  if $\varphi$ is a homeomorphism of   $\overline \Om$ onto $\overline \Om'$ and a conformal map of  $\Om$ onto $\Om'$.

Note that in the previous lemma  we necessarily have $f^n(\overline U)=\overline V$,  $\alpha(p_D)=p_V=f^n(p_U)$, and  $\alpha(\partial D)\sub \partial V$ by (iii).   In the conclusion of the lemma we obtain a lift $\tilde \alpha$ for a given map $\alpha$
under the  branched covering map $f^n$  so that    the  diagram
\begin{equation*}
  \xymatrix {
  &\overline U  \ar[d]^{f^n}   \\
    \overline{D}  \ar[ur]^{\tilde \alpha}  \ar[r]^{\alpha}  &   \overline V 
  }
\end{equation*}
commutes. 
By Lemma~\ref{lem:FatouDyn} the map  $f^n$ is actually an unbranched covering map from $\overline U\setminus \{p_U\}$ onto 
$\overline V\setminus \{p_V\}$. The lemma asserts that  the existence and uniqueness 
 of a lift $\tilde \alpha$  is guaranteed if the boundary map 
 $\alpha|_{\partial D}$ has a lift (namely $\beta$), and if we have some compatibility condition 
 for branch points (given by condition (ii)). 
 
 \begin{proof} The lemma easily follows from some basic theory for covering maps and lifts
 (see \cite{Ha} for general background), so we will only sketch the argument and leave some straightforward details to the reader.
  
 By Lemma~\ref{lem:deg} we can change $\overline U$ and $\overline V$ by conformal homeomorphisms so that we can assume $\overline U=\overline V=\overline \D$, $p_U=0=p_V$,  and $f^n=P_k$ for suitable $k\in \N$ without loss of generality. 
 By classical  conformal mapping theory we may also assume that $D=\D$ and $p_D=0$. 
 Then condition (ii) translates to $\alpha(0)=0$ and $\alpha(z)\ne 0$ for $z\in \overline \D\setminus \{0\}$. 

We use this to define a homotopy of the boundary map $\alpha|_{\partial \D}$ into the base space $\overline \D\setminus \{0\}$ of the covering map $P_k\: \overline \D\setminus \{0\}\ra \overline \D\setminus \{0\}$. Namely, let  $H\: \partial \D\times (0,1]\ra \overline \D\setminus \{0\}$ be defined as
$H( \zeta,t):=\alpha(t\zeta)$ for $\zeta\in \partial \D$ and $t\in(0,1]$.  It is convenient to think of $H$ as a homotopy running backwards in time $t\in (0,1]$ starting at $t=1$. 
Note that $P_k\circ\beta=\alpha|_{\partial \D}=H(\cdot, 1)$. So for the initial time $t=1$ the homotopy has the  lift $\beta$ under the covering map $P_k\: \overline \D\setminus \{0\}\ra \overline \D\setminus \{0\}$. By the homotopy lifting theorem \cite[p.~60, Proposition~1.30]{Ha}, the whole homotopy $H$ has a unique lift
starting at $\beta$, i.e., there exists a unique continuous map $\widetilde H\:  \partial \D\times (0,1]\ra \overline \D\setminus \{0\}$ such that 
$P_k\circ \widetilde H=H$ and $ \widetilde H(\cdot, 1)=\beta$. Now we define $\tilde \alpha (z)=\widetilde H( z/|z|, |z|)$ for $z\in   
 \overline \D\setminus \{0\}$. Then  $\tilde \alpha$ is continuous on 
 $\overline \D\setminus \{0\}$, where it satisfies $P_k\circ \tilde \alpha=\alpha$. Since $\alpha(0)=0$, this last equation  implies that we get a continuous extension 
 of $\tilde \alpha$ to $\overline \D$ by setting $\tilde \alpha(0)=0$. This extension is a lift $\tilde \alpha$ of $\alpha$. Note that  $\tilde \alpha^{-1}(0)=0$ and  
 $\tilde \alpha|_{\partial\D}=\widetilde H(\cdot, 1)=\beta$. Moreover, $\tilde \alpha$ is holomorphic on $\D$, because it is a continuous branch of the $k$-th root of the holomorphic function
 $\alpha$ on $\D$. This shows   that $\tilde \alpha$ has  the desired properties. The uniqueness of $\tilde \alpha$  easily follows from the uniqueness of $\widetilde H$.   
 
We have $\beta= \tilde \alpha|_{\partial \D}$;  so  if $\beta$  is a homeomorphism,
then the argument principle implies that $\tilde \alpha$ is a conformal homeomorphism of 
$\overline \D$ onto  $\overline \D$. 
 \end{proof}

\section{The conformal elevator for subhyperbolic maps}
\label{s:confelev} 
\no
A rational map $f$ is called {\em subhyperbolic} if each critical point of $f$ in $\Ju(f)$ has a finite orbit while each critical point  in $\Fa(f)$ has an orbit that  converges to  an attracting or superattracting cycle of  $f$. 
The map $f$ is called {\em hyperbolic} if it is subhyperbolic and $f$ does not have critical points in $\Ju(f)$. Note that every postcritically-finite rational map is subhyperbolic.

For the rest of  this section,  we will assume that $f$ is a subhyperbolic rational map with $\Ju(f)\ne \hC$. 
Moreover, we will make the following additional assumption:
\begin{equation}\label{eq:addinv}
\Ju(f)\sub \tfrac 12 \D\quad \text{and}\quad f^{-1}( \D)\sub  \D. 
\end{equation}
Here and in what follows, if $B$ is a disk, we denote by $\frac12 B$
the disk with the same center and whose radius is half the radius of $B$. 

The inclusions \eqref{eq:addinv} can always be achieved by conjugating $f$ 
with an appropriate M\"obius transformation so that $\Ju(f)\sub \tfrac 12\D$ and $\infty$ is an attracting or superattracting 
periodic point of $f$. If we then replace $f$ with suitable iterate, we may in addition 
assume that $\infty$ becomes an attracting or superattracting  fixed point of $f$ with 
$f(\hC\setminus \D)\sub\hC\setminus \D$. The latter inclusion is equivalent to 
$f^{-1}( \D)\sub  \D$. 


Every small disk  $B$ centered at a point in $\Ju(f)$ can be ``blown up" by a carefully chosen iterate $f^n$  
to a definite size with good control on  how sets  are distorted under the map $f^n$. We will discuss this in detail as a preparation for the proofs of  Theorems~\ref{thm:circgeom} and \ref{thm:main2}, and will  refer to this procedure as applying the {\em conformal elevator} to $B$.  In the following, all metric notions refer to the Euclidean metric on $\C$.

Let $P\sub \hC$ denote the union of all superattracting or attracting cycles of $f$. 
This is a non-empty and finite set contained in $ \Fa(f)$.  Since $f$ is subhyperbolic, every critical point in $\Ju(f)$ has a finite orbit, and every critical point in $\C \setminus \Ju(f)$ has an orbit that converges to $P$. 
Hence there exists a neighborhood of $\Ju(f)$ that  contains only finitely many points in $\post (f)$ and no points in $\post^c(f)$. This implies that we can choose  $\eps_0>0$ so  small that  
$\diam(\Ju(f))> 2\eps_0$, and so that 
every disk $B'=B(q,r')$  centered at a point $q\in \Ju(f)$ with positive radius 
$r'\le 8\eps_0$ is contained in $\D$, contains no point in $\post^c(f)$ and at most one point in 
$\post (f)$.

Let $B=B(p,r)$ be a small  disk centered at a point $p\in \Ju(f)$ and of positive radius $r<\eps_0$. Since $B$ is centered at a point in $\Ju(f)$, we have $\Ju(f)\sub f^n(B)$ for sufficiently large $n$ (see \cite[p.~69, Theorem~4.2.5~(ii)]{Be}), and so the images of $B$ under iterates will eventually have  diameter $>2\eps_0$. Hence there exists a maximal number $n\in \N_0$ such that $f^n(B)$ is contained in the disk of radius $\eps_0$ centered at a point $\tilde q\in \Ju(f)$. 

 If $B(\tilde q, 2\eps_0)\cap \post (f)=\emptyset$, we define $q=\tilde q$ and 
$B'=B(q, 2\eps_0)$. 
Otherwise, there exists a unique point $ q\in B(\tilde q, 2\eps_0)\cap \post(f)$.
Then we define 
$$B':=B( q, 8\eps_0)\supset B(q, 4\eps_0)\supset  B(\tilde q, 2\eps_0)\supset  f^n(B). $$
In both  cases,  we have 

\begin{itemize}

\smallskip 
\item[{(i)}]
$f^n(B)\sub \frac 12B'\sub \D$, 

\smallskip 
\item[{(ii)}]
$B'\cap \post^c(f)=\emptyset$, 

\smallskip 
\item[{(iii)}]
  $\#(B'\cap \post(f))\le 1$ with equality only if $B'$ is centered at a point in $\post(f)$. 
\end{itemize}
 
By definition of $n$, the set $f^{n+1}(B)$ must have diameter $\ge \eps_0$. Hence by uniform continuity of $f$ near $\Ju(f)$ there exists $\delta_0>0$ independent of $B$ such that 
\begin{itemize}

\smallskip 
\item[{(iv)}] $\diam (f^n(B))\ge \delta_0$. 

\end{itemize}

Let $\Om\sub \hC$ be the unique component of $f^{-n}(B')$ that contains 
$B$. Then by Lemma~\ref{lem:deg} and by \eqref{eq:addinv}, 
\begin{itemize}

\smallskip 
\item[{(v)}]
$\Om$ is simply connected,  and $B\sub \Om \sub \D$,

\smallskip 
\item[{(vi)}]
the map $f^n|_{\Om}\: \Om \ra B'$ is proper, 

\smallskip 
\item[{(vii)}]
there exists $k\in \N$,  and   conformal maps
$\varphi\: B'\ra \D$ and $\psi\: \Om\ra \D$ such that 
$$(\varphi\circ f^n\circ \psi^{-1})(z)= z^k$$ 
for all $z\in \D$. Here $k\in \N$ is uniformly bounded independent of $B$. 
\end{itemize}

If $k\ge 2$, then $q=\varphi^{-1}(0)\in \post(f)\cap B'$, and so $q$ is the center of $B'$. If $k=1$, then we can choose $\varphi$ so that this is also the case.
So 

\begin{itemize}

\smallskip 
\item[{(viii)}] $\varphi$ maps the center $q$ of $B'$ to $0$. 

\end{itemize}

We refer to the choice of  $f^n$  and  the associated sets   $B'$ and  $\Om$ and the maps
 $\varphi$ and $\psi$
satisfying properties (i)--(viii) as {\em applying the  conformal elevator} to $B$.

\begin{lemma} \label{lem:dist}
There exist constants $\gamma,r_1>0$ and $C_1,C_2,C_3\ge 1$
 independent of $B=B(p,r)$ 
with the following properties:

\begin{itemize} 

\smallskip 
\item[{(a)}] If $A\sub B$ is a connected set, then
$$\frac{\diam(A)}{\diam(B)}\le C_1 \diam(f^n(A))^{\gamma}. $$

\smallskip 
\item[{(b)}] $B(f^n(p), r_1)\sub f^n(\frac 12 B)\sub f^n(B)$.

\smallskip 
\item[{(c)}] If $u,v\in B$, then $$|f^n(u)- f^n(v)|\le C_2\frac{|u-v|}{\diam(B)}. $$ 

\smallskip 
\item[{(d)}] If $u,v\in B$, $u\ne v$,  and $f^n(u)=f^n(v)$, then  the center $q$ of $B'$ belongs to $\post(f)$ and we have $$|f^n(u)- q| \le C_3
\frac{|u-v|}{\diam(B)}. $$ 

\end{itemize} 
\end{lemma}

So (a)  says that a connected set $A\sub B$  comparable in diameter to $B$  
is blown up to a definite size under the conformal elevator, and by (b) the image of $B$ contains a disk of a definite size. If we consider the maps  $f^n|_B$ for different $B$, then by (c) they  are uniformly Lipschitz if we rescale distances in $B$ by $1/\diam(B)$. In (d) the center $q$ of $B'$ must be a point in $\post(f)$ for otherwise $f^n$ would be injective;  so  (d) says that if  distinct, but nearby points are mapped to the same image  $w$ under $f^n$, then a postcritical point must be close to this image $w$.

\begin{proof} In the following we write $a\lesssim b$ or $a\gtrsim b$ for two quantities $a,b\ge 0$ if we can find a constant $C>0$ independent of the disk $B$ such that $a\le C b$ or $Ca\ge b$, respectively. We write $a\approx b$ if we have both $a\lesssim b$ and $a\gtrsim b$, and in this case say that the quantities $a$ and $b$ are {\em comparable}.

We consider the conformal maps $\varphi\: B'\ra \D$ and 
$\psi\: \Om\ra\D$ satisfying  properties (vii) and (viii) of the conformal elevator as discussed above.  The exponent $k$ in (vii) is uniformly bounded, say $k\le N$, where $N\in \N$ is independent of $B$. As before, we use the notation $P_k(z)=z^k$ for $z\in \D$.  

As we will see, the properties (a)--(d) easily follow from distortion properties of the map $P_k$. We discuss the relevant  properties of $P_k$ first (the proof is left to the reader).
The map $P_k$ is Lipschitz with uniformly bounded Lipschitz constant, because $k$ is
uniformly bounded.
If $M\sub \D$ is connected, then 
$$ \diam (P_k(M))\gtrsim \diam(M)^k\gtrsim \diam(M)^N. $$
Moreover,  if $B(z, r_0)\sub \D$, then 
$$ B(P_k(z), r_1)\sub P_k(B(z, r_0)), $$ 
where $r_1 \gtrsim r_0^k \ge  r_0^N$. 

 By (vii) the map $\varphi$ is a Euclidean similarity, and so 
$\varphi(\frac 12 B')=\frac 12 \D$. 
Since the radius of $B'$ is  equal to $2\eps_0$ or $8\eps_0$, and hence comparable to $1$, we 
have 
\begin{equation} 
|\varphi(u')-\varphi(v')|\approx |u'-v'|, 
\end{equation}
whenever $u',v'\in B'$. 

 Moreover,  for  $\rho :=2^{-1/N}\in  (0,1)$ (which is independent of $B$)  we have  
 $$D:=B(0,\rho)\supseteq  P_k^{-1}( \tfrac 12 \D)=P_k^{-1}(\varphi(\tfrac 12 B'))=
 \psi(f^{-n}(\tfrac 12 B')\cap \Om). $$ Since 
 $f^n(B)\sub \frac 12 B'$  by 
 (i) and $B\sub \Om$ by (v), we then have $\psi(B) \sub D$. 
 
 So if  $u,v\in B$, then $\psi (u), \psi(v)\in D$. Hence by  the Koebe distortion theorem we have 
 $$|u-v| \approx |(\psi^{-1})'(0)|\cdot |\psi(u)-\psi(v)|$$
 whenever $u,v\in B$. 
 In particular, 
 $$ \diam(B)\approx  |(\psi^{-1})'(0)|\cdot \diam (\psi(B)) $$
 On the other hand, by (iv)
 \begin{align*}
 1&\approx\diam (f^n(B))\approx  \diam \big(\varphi (f^n(B))\big) \\ &=    \diam \big(P_k(\psi(B))\big) \lesssim \diam (\psi(B))\le 2. 
 \end{align*}
 Hence $\diam (\psi(B)) \approx 1$,   and so $\diam(B)\approx  |(\psi^{-1})'(0)|$.
 This implies that 
 \begin{equation}\label{eq:adKoe}
 \frac { |u-v|}{\diam(B)}\approx |\psi(u)-\psi(v)|, 
 \end{equation}
 whenever $u,v\in B$. 
 
 Now let $A\sub B$ be connected. Then $\psi(A)$ is connected, which implies 
  \begin{align*}
 \frac { \diam(A)}{\diam(B)}&\approx \diam(\psi(A))\lesssim     \diam \big(P_k( \psi(A))\big)^{1/N} \\ & =  \diam \big(\varphi(f^n(A))\big)^{1/N} \approx \diam (f^n(A))^{1/N}. 
 \end{align*}
 Inequality (a) follows.
 
It follows from  \eqref{eq:adKoe}  that 
there exists $r_0>0$ independent of $B$ such that 
$ B(\psi(p), r_0))\sub \psi(\frac 12 B)$. By the distortion property of $P_k$ mentioned in the beginning of the proof, $\varphi(f^n(\frac 12 B))=P_k(\psi(\frac 12 B))$ then contains a disk 
$B(P_k(\psi(p)), r_1)$ with $r_1>0$ independent of $B$. Since $\varphi(f^n(p))=P_k(\psi(p))$, and $\varphi$  distorts distances uniformly,  statement (b) follows. 

For (c) note that if $u,v\in B$, then 
 \begin{align*}
|f^n(u)-f^n(v)|&\approx |\varphi(f^n(u))-\varphi(f^n(v))|=   |P_k(\psi(u))-P_k(\psi(v))|\\
&\lesssim    |\psi(u)-\psi(v)| \approx  \frac { |u-v|}{\diam(B)}.   \end{align*}
We used that $P_k$ is Lipschitz on $\D$ with a uniform Lipschitz constant.

Finally we prove (d). If $u,v\in B$, $u\ne v$, and $f^n(u)=f^n(v)$, then $f^n$ is not injective on $B'$, and so the center $q$ of $B'$ belongs to $\post(f)$. Moreover, we then have
$\psi(u)^k=\psi(v)^k$, but $\psi(u)\ne \psi(v)$. This implies that 
$$ |\psi(u)-\psi(v)|\gtrsim \frac 1k |\psi(u)| \approx |\psi(u)|. $$
It follows that 
\begin{align*}
|f^n(u)-q|&\approx |\varphi(f^n(u))-\varphi(q)|= |\psi(u)^k|\\
&\le |\psi(u)| \lesssim    |\psi(u)-\psi(v)| \approx  \frac { |u-v|}{\diam(B)}.   \end{align*}

\end{proof}

\section{Geometry of the peripheral circles}
\label{s:geomper} 
\no In this section we will prove Theorem~\ref{thm:circgeom}. We have already defined in Section~\ref{ss:qc} what it means for the peripheral circles of a Sierpi\'nski carpet $S$ to be uniform  quasicircles and to be uniformly relatively separated. We say that 
the 
peripheral circles of  $S$  {\em  occur on all locations and scales} if  there exists a constant $C\geq 1$ such that for every $p\in S$ and every $0<r\leq {\rm diam}(\hC)=2$, there exists a peripheral circle $J$ of $S$ with  $B(p,r)\cap J\ne 0$ and 
$$
r/C\leq{\rm diam}(J)\leq C r.
$$ 
Here and below the  metric notions refer to the chordal metric
$\sigma$  on $\hC$.

A set $M\sub \hC$ is called {\em porous} if there exists a constant $c>0$ such that 
for every  $p\in S$ and every $0<r\leq 2$ there exists a point $q\in B(p,r)$ such that $B(q,cr)\sub \hC\setminus M$. 

Before we turn to the proof of Theorem~\ref{thm:circgeom}, we require an auxiliary fact.

\begin{lemma}\label{L:PerC} Let $f$ be a rational map such that $\Ju(f)$ is a Sier\-pi\'n\-ski carpet, and  let $J$ 
be a peripheral circle of $\Ju(f)$. Then  $f^n(J)$ is  a peripheral circle of $\Ju(f)$, and $f^{-n}(J)$ is a union of finitely many peripheral circles of $\Ju(f)$ for each $n\in \N$.
Moreover, $J\cap \post(f)=\emptyset=J\cap \crit(f)$.
%
%
%
\end{lemma}

\begin{proof} 
 There exists precisely one Fatou component $U$ of $f$  such that $\partial U=J$. Then $V=f^n(U)$ is also a Fatou component  of $f$. Hence $\partial V$ is a peripheral circle of $\Ju(f)$. The map $f^n|_U\: U\ra V$ is proper which implies that 
 $f^n(J)=f^n(\partial U)=\partial V$. Similarly, there are finitely many distinct Fatou components $V_1, \dots, V_k$ of $f$ such that 
$$ f^{-n}(U)=V_1\cup \dots \cup V_k. $$
Then 
$$f^{-n}(J)=\partial V_1\cup \dots \cup \partial V_k, $$ 
and so the preimage of $J$ under $f^n$  consists of the finitely many disjoint Jordan curves 
$\partial V_i$, $i=1, \dots, k$,  which are peripheral circles of $\Ju(f)$.

To show $J\cap \post(f)=\emptyset$, we argue by contradiction, and assume that there exists 
a point $p\in \post(f)\cap J$. Then there exists $n\in \N$, and $c\in \crit(f)$ such that 
$f^n(c)=p$. As we have just seen,   the preimage of $J$ under $f^n$  consists of finitely many disjoint Jordan curves, and is hence a topological $1$-manifold.  On the other hand, since $c\in f^{-n}(p)\sub f^{-n}(J)$ is a critical point  of $f$ and hence of $f^n$, at  $c$ the set 
$f^{-n}(J)$ cannot be a $1$-manifold. This is a contradiction. 

Finally, suppose that $c\in J\cap \crit(f)$. Then $f(c)\in \post(f)\cap f(J)$, and 
$f(J)$  is a peripheral circle of $\Ju(f)$. This is impossible by what we have just seen. 
\end{proof} 

\begin{proof}[Proof of Theorem~\ref{thm:circgeom}.] A general idea for the proof is to argue by contradiction, and get locations  where the desired statements fail quantitatively in a worse and worse manner. One can then use the dynamics to blow up to a global scale and derive a contradiction from topological facts. It is fairly easy to implement this idea if we have expanding  dynamics  given by a group (see, for example,  \cite[Proposition~1.4]{Bo}). In the present case, one applies the conformal elevator and the estimates as given by Lemma~\ref{lem:dist}.
We now provide the details.  

We can pass to iterates of the map $f$, and also conjugate $f$ by a M\"obius transformation
as properties that we want to establish are M\"obius invariant. This M\"obius invariance 
is explicitly stated for peripheral circles to be uniform quasicircles and to be uniformly 
relatively separated in \cite[Corollary 4.7]{Bo}. The M\"obius invariance of the other stated  properties  immediately follows from the fact that 
each M\"obius transformation is  bi-Lipschitz with respect to the chordal metric.
In this way, we may assume that  \eqref{eq:addinv} is true. Then the peripheral circles are subsets of $\D$, where chordal and Euclidean metric are comparable. Therefore, we can use the Euclidean metric, and all metric notions will refer to this metric in the following.  

\medskip
{\em Part I.} To show  that  peripheral circles of $ \Ju(f)$ are uniform quasicircles, we argue by contradiction. 
Then for each $k\in \N$ there exists a  peripheral circle $J_k$ of $\Ju(f)$, and  distinct points $u_k, v_k\in J_k$ such that if $\alpha_k,
\beta_k$ are the two subarcs of $J_k$ with endpoints $u_k$ and $v_k$, then 
\begin {equation}\label{eq:userel}
 \frac {\min\{\diam (\alpha_k), \diam (\beta_k)\}} {|u_k-v_k|} \to \infty
 \end{equation} 
as $k\to \infty$.
We can pick $r_k>0$ such that 
\begin{equation} \label{eq:diam/r}\min\{\diam (\alpha_k), \diam (\beta_k)\}/r_k\to \infty
\end{equation} 
and 
\begin{equation} \label{eq:dist/r} |u_k-v_k|/r_k\to 0
\end{equation} 
as $k\to \infty.$
We now apply the conformal elevator to $B_k:=B(u_k, r_k)$. Let $f^{n_k}$ be the corresponding iterate and $B_k'$ be the ball as discussed in Section~\ref{s:confelev}. 
Define $J'_k=f^{n_k}(J_k)$,  $u'_k=f^{n_k}(u_k)$, and $v'_k=f^{n_k}(v_k)$.
Then  Lemma~\ref{lem:dist}~(a) and \eqref{eq:diam/r} imply that the diameters of the sets 
 $J'_k$
are uniformly bounded away from $0$ independently of $k$. 
Since $J_k'$ is a peripheral circle of the Sierpi\'nski carpet $\Ju(f)$ by Lemma~\ref{L:PerC}, there are only finitely many possibilities for the set $J'_k$. By passing to suitable subsequence if necessary, we may assume that $J'=J'_k$ is a fixed peripheral circle of  $\Ju(f)$ independent of $k$.    
The points $u'_k,v'_k$ lie in $J'$ and  by \eqref{eq:dist/r} and  Lemma~\ref{lem:dist}~(c) we have 
\begin{equation}\label{eq:diamto0}
|u'_k-v'_k|\to 0
\end{equation} 
as $k\to \infty$. 

For large $k$ we want to find a point $w_k\ne u_k$ in $B_k$ near $u_k$ with 
$f^{n_k}(u_k)=f^{n_k}(w_k)$. If $u'_k=v'_k$ we can take $w_k=v_k$. Otherwise, if $u'_k\ne v'_k$,  there are two subarcs of $J'$ with endpoints 
$u'_k$ and $v'_k$. Let  $\ga'_k\sub J'$ be the one with smaller diameter.
Then by \eqref{eq:diamto0} we have
\begin{equation}\label{eq:ga_kprimeto0}
\diam(\ga'_k)\to 0
\end{equation}
as $k\to \infty$ (for the moment we only consider such  $k$ for which $\ga'_k$ is defined).

Since $J'\cap \post(f)=\emptyset$ by Lemma~\ref{L:PerC}, the map $f^{n_k}\: J_k \ra J'$ is a covering map. So we can lift the arc  $\ga'_k$ under $f^{n_k}$ to a subarc $\ga_k$ 
of $J_k$ with initial point $v_k$ and $f^{n_k}(\ga_k)=\ga'_k$.  By Lemma~\ref{lem:dist}~(b) we have  $\ga'_k\sub f^{n_k} (B_k)$ for large $k$;  then  Lemma~\ref{lem:dist}~(a) implies that    $\ga_k\sub B_k$ for large $k$, and also 
\begin{equation}\label{eq:ga_kto0}
\diam(\ga_k)/r_k\to 0
\end{equation}
 as $k\to \infty$.  Note that if $w_k$ is the other endpoint of $\ga_k$, then $f^{n_k}(w_k)=u'_k$.  We have $w_k\ne u_k$ for large $k$;
  for if $w_k=u_k$, then $\ga_k\sub J_k$ has the endpoints 
  $u_k$ and $v_k$ and so must agree with one of the arcs $\alpha_k$ or $\beta_k$; but for large $k$ this is impossible by \eqref{eq:diam/r}
  and \eqref{eq:ga_kto0}. In addition, we  have 
  $$|u_k-w_k|/r_k\le |u_k-v_k|/r_k+\diam(\ga_k)/r_k\to 0$$
  as $k\to \infty$. Note that this is also true if $w_k=v_k$.
  
  In summary, for each large $k$ we can find a point $w_k\in B_k$ with 
  $w_k\ne u_k$, $f^{n_k}(u_k)=f^{n_k}(w_k)$, and 
  \begin{equation} \label{eq:distukwk}
  |u_k-w_k|/r_k\to 0
  \end{equation}
   as 
  $k\to \infty$.
  
  Then by Lemma~\ref{lem:dist}~(d)  the center $q_k$ of $B'_k$ must belong to the 
postcritical set of $f$ and 
$$\dist(J', \post(f))\le  |u'_k- q_k|\to 0$$
as $k\to \infty$. Since $f$ is subhyperbolic, every sufficiently small neighborhood of $\Ju(f) \supseteq J'$ contains only finitely many points in $\post(f)$, and 
so this implies $J'\cap \post(f)\ne \emptyset$.  We know that this is impossible by Lemma~\ref{L:PerC} and so we  get a contradiction. This shows that the peripheral circles are uniform quasicircles.

\medskip
{\em Part~II.}
The proof that the peripheral circles of $\Ju(f)$ are uniformly relatively separated runs along almost identical lines. Again we argue by contradiction.
Then for $k\in \N$ we can find distinct peripheral circles $\alpha_k$ and $\beta_k$ of $\Ju(f)$,  and points $u_k\in \alpha_k$, $v_k\in \beta_k$ such that  
\eqref{eq:userel} is valid. We can again pick $r_k>0$ so that the relations 
\eqref{eq:diam/r} and \eqref{eq:dist/r} are true. As before we define $B_k=B(u_k, r_k)$ and apply the conformal elevator to $B_k$ which gives us suitable iterate $f^{n_k}$ and a ball $B'_k$. By Lemma~\ref{lem:dist}~(a) the images of 
$\alpha_k$ and $\beta_k$ under $f^{n_k}$ are blown up to a definite size.
Since there are only finitely many peripheral circles of $\Ju(f)$ whose diameter exceeds a given constant, only finitely many such image pairs can arise.
By passing to a suitable subsequence if necessary, we may assume that 
$\alpha=f^{n_k}(\alpha_k)$ and $\beta=f^{n_k}(\alpha_k)$ are peripheral circles independent of $k$. 

We define  $u'_k:=f^{n_k}(u_k)\in \alpha$ and $v'_k:=f^{n_k}(v_k)\in \beta$. Then again the relation \eqref{eq:diamto0} holds. This is only possible 
if $\alpha\cap \beta\ne \emptyset$, and so $\alpha=\beta$. 

Again for large $k$ we want to find a point $w_k\ne u_k$ in $B_k$ near $u_k$ with 
$f^{n_k}(u_k)=f^{n_k}(w_k)$. If $u'_k=v'_k$ we can take $w_k:=v_k$. 
Otherwise, if $u'_k\ne v'_k$,  we let $\ga'_k$ be the subarc of 
$\alpha=\beta$ with endpoints $u'_k $ and  $v'_k$ and smaller diameter.
Then 
we can lift $\ga'_k$ to a subarc $\ga_k\sub \beta_k$ with initial point $v_k$ 
such that $f^{n_k}(\ga_k)=\ga'_k$, and  we have \eqref{eq:ga_kto0}. If $w_k\in \beta_k$ is the other endpoint of 
$\ga_k$, then $f^{n_k}(w_k)=u'_k=f^{n_k}(u_k)$, and $w_k\ne u_k$, because these points lie in the disjoint sets $\beta_k$ and $\alpha_k$, respectively.  Again we have \eqref{eq:distukwk},  which implies that the center $q_k$ of $B_k'$ belongs to $\post(f)$, and leads 
to $\dist(\alpha, \post(f))=0$. We know that this is impossible by Lemma~\ref{L:PerC}.  

\medskip
{\em Part~III.} We will show that peripheral circles of $\Ju(f)$ appear on all locations and scales. 


Let $p\in \Ju(f)$ and $r>0$ be arbitrary, and define $B=B(p,r)$. We may assume that $r$ is small, because by a simple compactness argument one can show that 
 disks of definite, but not too large Euclidean size contain  peripheral circles of comparable diameter. 
  
 We now apply the conformal elevator to $B$ to obtain  an iterate $f^n$.
Lemma~\ref{lem:dist}~(b) implies that there exists a fixed constant $r_1>0$ independent of $B$ such 
that $B(f^n(p), r_1)\sub f^n(\frac 12 B)$. By part (a) of the same lemma, we can also find a constant 
$c_1>0$ independent of $B$ with the following property: if $A$ is a connected set with 
$A\cap B(p, r/2)\ne \emptyset$ and $\diam(f^n(A))\le c_1$, then $A\sub B$.

We can now find a peripheral circle $J'$ of $\Ju(f)$ such that 
$J'\cap f^n(\frac12 B)\ne \emptyset$ and $0<c_0 < \diam(J')<c_1$, where $c_0$ is another positive constant independent of $B$. This easily follows from a compactness argument based on the fact that $f^n(\frac12 B)$ contains  a disk of a definite size that is centered at a point in $\Ju(f)$. 

The  preimage $f^{-n}(J')$ consists of finitely many components that are peripheral circles of $\Ju(f)$. One of these peripheral circles $J$ meets $\frac12 B$. Since $\diam(f^n(J))=\diam(J')<c_1$, by the  choice of $c_1$ we then have $J\sub B$, and so $\diam(J)\le 2r$.
Moreover, it follows from Lemma~\ref{lem:dist}~(c) that 
$\diam(J)\ge  c_2 \diam(J')\diam(B)\ge c_3 r$, where again $c_2,c_3>0$ are independent of $B$.  The claim follows. 

\medskip
{\em Part~IV.} Let $p\in \Ju(f)$ be arbitrary and $r\in (0,1]$. To establish the porosity of $\Ju(f)$, it is enough to show 
that the Euclidean disk $B(p,r)$ contains a disk of comparable radius that lies in the complement of $\Ju(f)$. By what we have just seen, $B(p,r)$ contains a peripheral circle 
$J$ of diameter comparable to $r$. By possibly allowing a smaller constant of comparability,  we may assume that $J$ is distinct from the one peripheral circle $J_0$ that bounds the unbounded Fatou component  of $f$. Then $J\sub B(p,r)$   is the boundary of   a bounded Fatou component $U$, and so  $U\sub B(p,r)$. Since the peripheral circles of $\Ju(f)$ are uniform quasicircles, it follows that $U$ contains a Euclidean disk
$D$ of comparable size (for this standard fact see 
\cite[Propsition 4.3]{Bo}). 
Then $\diam (D) \approx \diam (J)\approx r$. Since $D\sub U\sub  B(p,r)\cap \hC\setminus \Ju(f)$ the porosity of $\Ju(f)$ follows.

 Finally, the porosity of $\Ju(f)$ implies that $\Ju(f)$ cannot have Lebesgue density points, and is hence a set of measure zero.
 \end{proof}

\section{Relative Schottky sets and Schottky maps}\label{s:Schottky}
\no
A
\emph{relative Schottky set} $S$ in a region  $D\subseteq\hC$ is a subset of $D$ whose
complement in $D$ is a union of open geometric  disks $\{B_i\}_{i\in I}$ with closures $\overline B_i,\ i\in I$, in $D$, and such that $\overline B_i\bigcap\overline B_j=\emptyset,\ i\neq j$. 
We write
\begin{equation}\label{eq:defrelSch}
S=D\setminus \bigcup_{i\in I}B_i.
\end{equation}
If $D=\hC$ or $\C$, we say that $S$ is a \emph{Schottky set}. 


Let $A,B\sub \hC$ and $\varphi\: A\ra B$ be a continuous map. We call $\varphi$ a {\em local homeomorphism} of $A$ to $B$ if for every point $p\in A$ there exist open sets $U,V\sub \C$ with   $p\in U$, $f(p)\in V$ such that $f|_{U\cap A}$ is a homeomorphism of $U\cap A$ onto $V\cap B$. Note that this concept depends of course on $A$, but also crucially  on $B$: if $B'\supseteq B$, then we may consider a local homeomorphism $f\:A\ra B$ also as a map $f\:A\to B'$, but the second map will not be a local homeomorphism in general.

Let $D$ and $\tilde D$ be two regions in $\hC$, and let $S=D\setminus \bigcup_{i\in I} B_i$ and $\tilde S=\tilde D\setminus \bigcup_{j\in J}\tilde B_j$ be relative Schottky sets in $D$ and 
$ \tilde D$, respectively.
Let $U$ be an open subset of $D$ and let $f\: S\cap U\to \tilde S$ be a local homeomorphism. 
According to \cite{Me3}, such a map $f$ is called a \emph{Schottky map}  if 
it is  conformal at every point $p\in S\cap U$, i.e., the derivative
\begin{equation}\label{eq:confmean}
f'(p)=\lim_{q\in S,\, q\to p }\frac{f(q)-f(p)}{q-p}
\end{equation}
exists and does not vanish, and the function $f'$ is  continuous  on $S\cap U$.
If $p=\infty$ or $f(p)=\infty$, the existence of this limit and the continuity of $f'$ have to be understood after a coordinate change  $z\mapsto 1/z$ near $\infty$. In all our applications
$S\sub \C$ and so we can  ignore this technicality. 

Theorem~\ref{T:Me2T1.2} implies that if $D$ and $\tilde D$ are Jordan regions, the relative Schottky set $S$ has measure zero, and $f\: S\to\tilde S$ is a locally quasisymmetric
homeomorphism that is 
 orientation-preserving (this is defined similarly as for homeomorphisms between Sierpi\'nski carpets; see the discussion after Lemma~\ref{lem:ext}), then $f$ is a Schottky map.
 
 We require  a more general criterion for maps to be Schottky maps. 



\begin{lemma}\label{L:Schmaps} Let $S\sub \C$ be a Schottky set of measure zero. 
Suppose   $U\sub \hC$ is  open and   $\varphi\: U\ra \hC$ is  a locally quasiconformal map with $\varphi^{-1}(S)=U\cap S. $
Then $\varphi\: U\cap S\ra S$ is a Schottky map. 

In particular, if $\psi\: \hC\ra \hC$ is a quasiregular map with $\psi^{-1}(S)=S$, then 
$\psi\:S\setminus \crit(\psi)\ra S$ is a Schottky map. 
\end{lemma}

In the statement the assumption $S\sub \C$  (instead of $S\sub \hC$) is not really essential, but helps to avoid some technicalities caused by the point $\infty$.

\begin{proof} Our assumption  $\varphi^{-1}(S)=U\cap S$ implies that 
$\varphi(U\cap S)\sub S$. So we can consider the restriction of $\varphi$ to $U\cap S$ as a map $\varphi\: U\cap S\ra S$ (for simplicity we do not use our usual notation $\varphi|_{U\cap S}$
for this and other restrictions in the proof). This map is a local homeomorphism  $\varphi\: U\cap S\ra S$.
Indeed, let $p\in U\cap S$ be arbitrary. Since $\varphi\: U\ra \hC$ is a local homeomorphism,
there exist  open sets $V,W\sub\hC$ with $p\in V\sub U$ and $f(p)\in W$ such that 
$\varphi$ is a homeomorphism of $V$ onto $W$. Clearly, $\varphi(V\cap S)\sub W\cap S$.
Conversely, if  $q\in  W\cap S$, then there exists a point $q'\in V$ with $\varphi(q')=q$; since 
$\varphi^{-1}(S)=U\cap S$, we have $q'\in S$ and so $q'\in V\cap S$. Hence 
$\varphi(V\cap S)=W\cap S$, which implies that $\varphi$ is a homeomorphism of $V\cap S$ onto $W\cap S$. 

Note that $p\in U\cap S$ lies on a peripheral circle of $S$ if and only if $\varphi(p)$ lies on a peripheral circle of $S$. Indeed, a point $p\in S$ lies on a peripheral of $S$ if and only if it is accessible by a path in the complement of $S$, and it is clear this condition is satisfied for a point $p\in S\cap U$ if and only if it is true for the image $\varphi(p)$ (see  \cite[Lemma~3.1]{Me3} for a more general related statement).



We now want to verify the other conditions for $\varphi$ to be a Schottky map based on 
Theorem~\ref{T:Me2T1.2}.  It is enough to reduce to this situation locally near each point $p\in U\cap S$.
We  consider two cases depending on whether $p$ belongs to a peripheral circle of $S$ or not. 

So suppose $p$ does not belong to any of the peripheral circles of $S$. Then there exist  arbitrarily small Jordan regions $D$ with $p\in D$ and $\partial D\sub S$ such that $\partial D$ does not meet any peripheral circle of $S$. This easily follows from   the fact that if we collapse each closure of a complementary component of $S$ in $\hC$ to a point, then the resulting quotient space is homeomorphic to $\hC$ by 
 Moore's theorem~\cite{Moo} (for more details on  this and the similar argument below, see the proof of \cite[Theorem~5.2] {Me3}). In this way we can find a small Jordan region $D$ 
with  the following properties: 

\begin{itemize}
\item[(i)] $p\in D \sub \overline D\sub U$, 

\smallskip
\item[(ii)] the boundary $\partial D$ is contained in $S$, but does not meet any peripheral circle 
of $S$,

\smallskip
\item[(iii)] $\varphi$ is a  homeomorphism of $\overline D$ onto the closure $\overline D'$ of another Jordan region 
$D'\sub \hC$. 
\end{itemize}

As in the first part of the proof,  we see that $\varphi$ is a homeomorphism 
of $D\cap S$ onto $D'\cap S$. This homeomorphism is  locally  quasisymmetric and orientation-preserving as it is the restriction  of a locally quasiconformal map. 
Since $\partial D$ does not meet peripheral circles of $S$, the same is true 
of its image $\partial D'=\varphi(\partial D)$ by what we have seen above.   It follows that the sets $D\cap S$ and $D'\cap S$ are   relative Schottky sets of measure zero contained  in the Jordan regions $D$ and $D'$, respectively. 
Note that the set $D\cap S$ is obtained by deleting from $D$ the complementary disks  of $S$ that are contained in $D$, and $D'\cap S$ is obtained  similarly.   Now Theorem~\ref{T:Me2T1.2} implies that $\varphi\: D\cap S\ra D'\cap S$ is a Schottky map which implies that 
$\varphi\: U\cap S \ra S$ is a Schottky map near $p$.

For the other case, assume that $p$ lies on a peripheral circle of $S$, say $p\in \partial B$, where $B$ is one of the disks that form the complement of $S$.  The idea is to use a Schwarz reflection procedure to arrive at a situation similar to the previous case. This is fairly straightforward, but we will provide the details for sake of completeness. 

Similarly as before 
 (here we collapse all closures of complementary components of $S$ to points except $\overline B$),  we  find a Jordan region $D$ with the  following properties: 

\begin{itemize}

\smallskip
\item[(i)] $\overline D\sub U$ and $\partial D=\alpha\cup \beta$, where $\alpha $ and $\beta$ are two non-overlapping arcs with the same endpoints such that $\alpha\sub \partial B$,  $\beta \sub S$,  $p$ is an interior point of $\alpha$, and  no interior point of $\beta$ lies on a peripheral circle of $S$,

\smallskip
\item[(ii)] $\varphi$ is a  homeomorphism of $\overline D$ onto the closure $\overline D'$ of another Jordan region 
$D'\sub \hC$. 

\end{itemize}
Let $\alpha'=\varphi(\alpha)$. Then $\alpha$ is contained in a peripheral circle $\partial B'$ of 
$S$, where $B'$ is a suitable complementary disk of $S$. Note that $\beta'=\varphi(\beta)$ is an arc contained in $S$, has its endpoints in $\partial B'$, and no interior point of  
$\beta'$ lies on a peripheral circle of $S$. 

Let $R\: \hC\ra \hC$ be the reflection in $\partial B$, and  $R'\: \hC\ra \hC$ be the reflection in $\partial B'$. Define $\tilde S=S\cup R(S)$ and
$\tilde S'=S\cup R'(S)$. 

 Then $\tilde S$ and $\tilde S'$ are Schottky sets of measure zero,  
$\partial B\sub \tilde S$, $\partial B'\sub \tilde S'$, and $\partial B$ and $\partial B'$ do not meet any of the peripheral circles of $\tilde S$ and $\tilde S'$, respectively.  

 Let $\tilde D=D\cup \inte(\alpha)\cup R(D)$ and $\tilde D'=D'\cup \inte(\alpha')\cup  R'(D')$,
 where $\inte(\alpha)$ and $ \inte(\alpha')$ denote the set of interior points of the arcs $\alpha$ and $\alpha'$, respectively. Then $\tilde D$ and $\tilde D'$ are Jordan regions such that $p\in \tilde D$, $\partial \tilde D\sub \tilde S$, $\partial \tilde D'\sub \tilde S'$, and 
$\partial \tilde D$ and $\partial \tilde D'$ do not meet any  of the peripheral circles of $\tilde S$ and $\tilde S'$, respectively.  Hence $\tilde D\cap \tilde S$ and $\tilde D'\cap \tilde S'$ are relative Schottky sets of measure zero in $\tilde D$ and $\tilde D'$, respectively.

We define a map $\tilde \varphi\: \tilde D\ra \tilde D'$ by 
$$ \tilde \varphi(z)= \left\{ \begin{array} {cl}\varphi(z)& \text{for $z\in D\cup \inte(\alpha) $,}\\&\\
 (R'\circ \varphi\circ R)(z)&\text{for $z\in R(D)\cup \inte(\alpha)$.} \end{array}
 \right.
 $$
Note that this  definition is consistent on $\inte(\alpha)$, because $\varphi(\alpha)=\alpha'=\overline D'\cap R'(\overline D')$. It is clear that $\tilde \varphi\: \tilde D\ra \tilde D'$ is a homeomorphism.
Moreover,  since the circular arc $\alpha$ (as any set of $\sigma$-finite Hausdorff $1$-measure) is   removable for quasiconformal maps  \cite[Section 35]{Va}, the map
$\tilde \varphi$ is locally quasiconformal, and hence locally quasisymmetric and orientation-pre\-serving.  It is also straightforward to see from the definitions and the relation $\varphi^{-1}(S)=U\cap S$ that 
$\tilde \varphi^{-1}(\tilde S')= \tilde D \cap \tilde S$. Similarly as in the beginning of the proof
this implies that $\tilde \varphi\: \tilde D\cap 
 \tilde S \ra  \tilde D'\cap  \tilde S'$ is a homeomorphism. Since it is also a local quasisymmetry
 and orientation-preserving,  it follows again from  Theorem~\ref{T:Me2T1.2} that $\tilde \varphi\: \tilde D\cap 
 \tilde S \ra  \tilde D'\cap  \tilde S'$ is a Schottky map. Note that $\tilde D\cap S= D\cap S$,
  that on this set the maps $\tilde \varphi$ and $\varphi$ agree, and that $\varphi(\tilde D\cap S)\sub S$.
 Thus,  $\varphi\: \tilde  D\cap S\ra S$ is a Schottky map, and so $\varphi\: U\cap S\ra S$ is a Schottky map near $p$. 
 
 It follows that  $\varphi\: U\cap S\ra S$ is a Schottky map as desired.

The second part of the statement immediately follows from the first; indeed, $\crit(\psi)$
is a finite set and so $U=\hC\setminus 
\crit(\psi)$ is an open subset $\hC$ on which $\varphi=\psi|_{U}\:U\ra \hC$ is a locally quasiconformal
map. Moreover, $\varphi^{-1}(S)=\psi^{-1}(S)\cap U=S\cap U$. By the first part of the proof,
$\varphi$ and hence also $\psi$ (restricted to $U\cap S$) is a Schottky map of $U\cap S=S\setminus \crit(\psi)$ into $S$. 
\end{proof} 

A relative Schottky set as in \eqref{eq:defrelSch}
 is called {\em locally porous  at $p\in S$} if there exists a neighborhood $U$ of $p$, and   constants $r_0>0$ and  $C\ge 1$ such that for each  $q\in S\cap U$  and  $r\in(0, r_0]$
there exists $i\in I$ with $B_i\cap B(q,r)\ne \emptyset$  and  $r/C\le \diam(B_i) \le Cr$.
 The relative Schottky set $S$  is called {\em locally porous} if it is locally porous at every point 
 $p\in S$.  Every locally porous relative Schottky set has measure zero since it cannot have Lebesgue density points.

For Schottky maps on locally porous Schottky sets very strong rigidity and uniqueness statements are valid such as 
Theorems~\ref{T:Me3T5.3} and \ref{C:Me3C4.2} stated in the introduction. We will need another result of  a similar flavor. 

\begin{theorem}[Me3, Theorem~4.1] \label{thm:dervuniq}
 Let  $S$ be a locally porous relative Schottky set in a region $D\sub \C$, let 
 $U\sub\C$  be an open set such that $S\cap U$ is  connected, and
  $u\: S\cap U\ra S$ be a Schottky map.  Suppose that there exists a point $a\in S\cap U$ with
 $u(a)=a$ and $u'(a)=1$. Then $u=\id|_{S\cap U}$. 
\end{theorem}


\section{A functional equation in the unit disk}

\no As discussed in the introduction, for the proof of Theorem~\ref{thm:main2} 
we will establish a functional equation of  form \eqref{eq:funddynrel} for the maps in question. 
For postcritically-finite maps $f$ and $g$ this leads to strong conclusions  based on the following lemma.  
Recall that $P_k(z)=z^k$ for $k\in \N$. 

\begin{lemma}\label{L:Rot} Let $\phi\: \partial \D \ra \partial \D$ be  an orientation-preserving homeo\-morphism, and suppose that there   exist numbers $k,l,n\in \N$, $k\ge 2$, such that 
\begin{equation} \label{eq:basiceq}
(P_l\circ \phi)(z)= (P_n\circ \phi\circ P_k)(z) \quad \text{for $z\in \partial \D$}.
\end{equation}
Then $l=nk$  and there exists  $a\in \C$ with $a^{n(k-1)}=1$ such that $\phi(z)=az$ for all $z\in \partial \D$.
\end{lemma}
 This lemma implies that we can uniquely extend $\phi$ to  a conformal homeomorphism
  from $\overline \D$ onto itself. It is also  important that this extension preserves the basepoint $0\in \overline \D$.

\begin{proof} By considering topological degrees, one immediately sees that 
$l=nk$. So if we introduce the map $\psi:= P_n\circ \phi$, then \eqref{eq:basiceq} can be rewritten as 
\begin{equation} \label{eq:basiceq2} 
P_k\circ \psi=\psi\circ P_k \quad \text{on $\partial \D$}.  
\end{equation}
Here the map $\psi\: \partial \D \ra \partial \D$ has degree $n$. We claim that this in combination with \eqref{eq:basiceq2} implies that for a suitable constant $b$ we have $\psi(z)=bz^n$ for $z\in \partial \D$.

 Indeed, there exists a continuous function $\alpha\: \R \ra \R$ with $\alpha(t+2\pi)=\alpha(t)$ such that
$$ \psi(e^{ i t}) = \exp( i n t+ i\alpha(t)) \quad \text{for  $t\in \R$}. $$
By  \eqref{eq:basiceq2} we have 
$$ \exp( i k n t+ i k\alpha(t)) = (\psi(e^{ i t}))^k = \psi(e^{ i k t})=
\exp( i kn t+ i\alpha(kt))$$
for $t\in \R$. This implies that there exists a constant $c\in \R$ such that 
$$ \alpha(t)=\frac 1k \alpha(tk)+c  \quad \text{for  $t\in \R$}. $$
Since $\alpha$ is $2\pi$-periodic, the right-hand side of this equation is $2\pi/k$-periodic as a function of $t$. Hence $\alpha$ is $2\pi/k$-periodic. Repeating this argument, we see that $\alpha$ is $2\pi/k^m$-periodic for all $m\in\N$, and so has arbitrarily small periods (note that $k\ge 2$). Since $\alpha$ is continuous, it follows that $\alpha$ is constant. Hence $\psi(z)=bz^n$ for $z\in \partial \D$ with a suitable constant $b\in \C$.

It follows that 
$$ \psi(z)=b z^n=\phi(z)^n  \quad \text{for $z\in \partial \D$}.$$
Therefore,  $\phi(z)=az$ for  $z\in \partial \D$ with a  constant $a\in \C$, $a\ne 0$.
Inserting this expression for $\phi$ into \eqref{eq:basiceq} and using $l=nk$, we conclude that $a^{n(k-1)}=1$ as desired. 
  \end{proof}

  \section{Proof of Theorem~\ref{thm:main2}}\label{s:proof}
 
  The proof will be given in several steps.

\medskip
{\em Step~I.} We first fix the setup. We can freely pass to iterates of the maps $f$ or $g$, because this  changes neither their Julia sets nor their postcritical sets. We can also 
conjugate the maps by M\"obius transformations. Therefore, as in 
 Section~\ref{s:confelev},  we may  assume that $$
\Ju(f), \Ju(g)\sub \tfrac 12 \D \text{ and } f^{-1}( \D), g^{-1}(\D)\sub  \D.
$$ 
Moreover, without loss of generality, we may require that $\xi$ is orien\-ta\-tion-preserving, for otherwise we can conjugate $g$ by $z\mapsto\overline z$. 
Since the peripheral circles of $\Ju(f)$ and $\Ju(g)$ are uniform quasicircles,  by Theorem~\ref{T:BoP5.1} the map $\xi$ extends (non-uniquely) to a quasiconformal, and hence quasisymmetric, map of the whole sphere. Then $\xi(\Ju(f))=\Ju(g)$ and $\xi(\Fa(f))=\Fa(g)$.  Since $\infty$ lies in Fatou components of $f$ and $g$, we may also assume that $\xi(\infty)=\infty$ (this normalization ultimately depends on the fact that for every point $p\in \D$ there exists a quasiconformal homeomorphism $\varphi$ 
 on $\hC$ with $\varphi(0)=p$ that is the identity outside $\D$). 
 Then $\xi$ is a quasisymmetry of $\C$ with respect to the Euclidean metric. In the following, all metric notions will refer to this metric. 
Finally, we define $g_\xi=\xi^{-1} \circ g\circ  \xi$.



\medskip
{\em Step~II.} We now carefully choose a location for a ``blow-down" by 
branches  of $f^{-n}$ which will be compensated by  a ``blow-up" by iterates   of $g$ (or rather $g_\xi$). 

Since repelling periodic points of $f$ are dense in $\Ju(f)$ (see
 \cite[p.~148, Theorem~6.9.2]{Be}), we can find such a point 
$p$ in $\Ju(f)$  that does not lie in $\post(f)$. Let   $\rho>0$  be a small  positive number such that 
the disk $U_0:=B(p, 3\rho)\sub \D$  is disjoint from $\post(f)$.
Since $p$ is periodic,  there exists $d\in \N$ such that $f^d(p)=p$. Let $U_1\sub \D$ be the component of $f^{-d}(U_0)$  that contains $p$. Since $U_0\cap \post(f)=\emptyset$, the set $U_1$ is a simply connected region, and $f^{d}$ is a conformal map from $U_1$ onto $U_0$ as follows from Lemma~\ref{lem:deg}.   Then there exists a unique inverse branch $f^{-d}$ with $f^{-d}(p)=p$ that is a conformal map of $U_0$ onto $U_1$. Since $p$ is a repelling fixed point 
for $f$, it is an attracting fixed point for this branch $f^{-d}$. By possibly choosing a smaller radius $\rho>0$ in the definition of $U_0=B(p, 3\rho)$ and by passing to an iterate 
of $f^d$, we may assume that $U_1\sub U_0$ and that $\diam(f^{-n_k}(U_0))\to 0$ as $k\to \infty$.
Here  $n_k=dk$ for $k\in \N$, and $f^{-n_k}$ is  the branch 
obtained by iterating the branch $f^{-d}$ $k$-times. Note that $f^{-n_k}(p)=p$ and 
$f^{-n_k}$ is a conformal map of $U_0$ onto a simply connected region 
$U_k$. Then $p\in U_k\sub U_{k-1}$ for $k\in \N$,  and
$\diam(U_k)\to 0$ as $k\to \infty$. 

The choice of these inverse branches is  {\em consistent} in the sense that  we have 
 \begin{equation}\label{eq:cons}
 f^{n_{k+1}-n_k}\circ f^{-n_{k+1}}=f^{-n_k}
 \end{equation}
 on $B(p, 3\rho)$ for all $k\in \N$. 
 Note that this consistency condition remains valid if we replace the original sequence $\{n_k\}$ by a subsequence.

Let $\tilde r_k>0$ be the smallest  number such that 
$$f^{-n_k}(B(p, 2\rho))\sub \tilde B_k:=B(p,\tilde  r_k).$$
Since $p\in f^{-n_k}(B(p, 2\rho))$ we have $\diam\big(f^{-n_k}(B(p, 2\rho))\big)\approx  \tilde  r_k$.
Here and below $\approx$   indicates  implicit positive multiplicative 
constants independent  of $k\in \N$. It follows that $\tilde r_k\to 0$ as $k\to \infty$. Moreover, since 
$f^{-n_k}$ is conformal on the larger disk $B(p,3\rho)$, Koebe's distortion theorem implies that 
$$ \diam\big(f^{-n_k}(B(p, 2\rho))\big)\approx \diam\big(f^{-n_k}(B(p, \rho))\big). $$

 Let $r_k>0$ be the smallest  number such that 
$\xi(\tilde B_k)\sub B_k:=B(\xi(p), r_k)$. Again  $r_k\to 0$ as $k\to \infty$ by   continuity 
of $\xi$.
We now want to apply the conformal elevator given by iterates of $g$ to the disks $B_k$.
For this we choose   $\epsilon_0>0$ for the map $g$ as in  Section~\ref{s:confelev}.

By applying the conformal elevator as described in Section~\ref{s:confelev}, we can find iterates
$g^{m_k}$ such that $g^{m_k}(B_k)$ is blown up to a definite, but not too large size,  and so 
$\diam(g^{m_k}(B_k))\approx 1$.

\medskip
{\em Step~III.} Now we consider the composition 
$$
h_k=g_\xi^{m_k}\circ f^{-n_k}= \xi^{-1}\circ g^{m_k}\circ \xi \circ f^{-n_k}
$$ 
defined on $B(p,2\rho)$ for $k\in \N$. We want to show that  this sequence subconverges 
locally uniformly on $B(p,2\rho)$ to a (non-constant) quasiregular  map $h\:B(p, 2\rho)\ra \C$.

 Since $f^{-n_k}$ maps $B(q,2\rho)$ conformally into 
$\tilde B_k$, $\xi$ is a quasiconformal map with $\xi(\tilde B_k)\sub B_k$, and  $g^{m_k}$ is holomorphic  on $\tilde B_k$, we conclude that the maps $h_k$ are uniformly quasiregular 
on $B(q,2\rho)$, i.e., $K$-quasiregular with $K\ge 1$ independent of $k$. The images 
$h_k(B(q,2\rho))$   are contained in a small Euclidean neighborhood of $\Ju(g)$ and 
hence in a fixed compact subset of $\C$. Standard 
convergence results for $K$-quasiregular mappings \cite[p.~182, Corollary 5.5.7]{AIM} imply that the sequence $\{h_k\}$
subconverges locally uniformly on  $B(q,2\rho)$ to a map $h\:B(q,2\rho)\ra \C$  that 
is also quasiregular, but possibly constant.  
 By passing to a subsequence if necessary, we may assume that 
$h_k\to h$ locally uniformly on $B(q,2\rho)$. 

To rule out  that $h$ is constant, it is enough  to show that for smaller 
 disk $B(p,\rho)$ there exists $\delta>0$ such that 
 $\diam\big(h_k(B(p,\rho))\big)\geq \delta$ for all $k\in \N$.
 
 We know that 
 $$\diam\big(f^{-n_k}(B(p,\rho))\big)\approx \diam\big(f^{-n_k}(B(p,2\rho))\big) \approx \diam(\tilde B_k). $$
Moreover, since $\xi$ is a quasisymmetry and $f^{-n_k}(B(p,\rho))\sub \tilde B_k$, this implies  
$$\diam\big(\xi(f^{-n_k}(B(p,\rho)))\big)\approx \diam (\xi(\tilde B_k))\approx \diam(B_k). $$
So the connected set $\xi(f^{-n_k}(B(p,\rho)))\sub B_k$ is    comparable in size to $B_k$. By 
Lemma~\ref{lem:dist}~(a)  the conformal elevator blows it up to a definite, but not too large size, i.e., 
$$
\diam\big((g^{m_k}\circ \xi\circ f^{-n_k})(B(p,\rho))\big)\approx 1.
$$
Since the sets  $(g^{m_k}\circ \xi\circ f^{-n_k})(B(p,\rho))$ all meet $\Ju(g)$, they stay in a 
compact part of $\C$, and so we still get a uniform lower bound for the diameter of these sets if we apply the homeomorphism $\xi^{-1}$; in other words,
$$\diam\big((\xi^{-1}\circ g^{m_k}\circ \xi\circ f^{-n_k})(B(p,\rho))\big)=\diam
 \big(h_k(B(p,\rho))\big)\approx 1$$
as claimed. We conclude that $h_k\to h$ locally uniformly on $B(p, 2r)$, where $h$ is non-constant and quasiregular. 

The quasiregular map $h$ has at most  countably many critical points, and so there exists a point $q\in B(p, 2\rho)\cap\Ju(f)$
and a radius $r>0$  such that  $B(q, 2r)\sub B(p, 2\rho)$ and $h$ is injective on $B(q, 2r)$ and hence quasiconformal.
Standard topological degree arguments imply  that at least on the smaller disk $B(q,r)$ 
the maps $h_k$ are also injective and hence quasiconformal  for all $k$ sufficiently large. By possibly disregarding finitely of the maps $h_k$, we may assume that $h_k$ is quasiconformal on $B(q,r)$ for all $k\in \N$.

To summarize, we have found a disk $B(q,r)$ centered at a point $q\in \Ju(f)$ such that 
the maps $h_k$ are defined and quasiconformal on $B(q,r)$ and converge uniformly on 
$B(q,r)$ to a quasiconformal map $h$. 

From the invariance properties of Julia  and Fatou sets and the mapping properties of $\xi$, it follows that 
$$ h_k(B(q, r)\cap \Ju(f))\sub \Ju(f) \text{ and } h_k(B(q, r)\cap \Fa(f))\sub \Fa(f)$$
for each $k\in \N$.
Hence 
\begin{equation}\label{eq: invarhk}
h_k^{-1}(\Ju(f))=\Ju(f)\cap B(q,r)
\end{equation}
for each map $h_k\: B(q,r)\ra \C$, $k\in \N$.  

Since $\Ju(f)$ is closed and $h_k\to h$ uniformly on $B(q, r)$, we also have $h(B(q, r)\cap \Ju(f))\sub \Ju(f)$. To get a similar inclusion relation also for the Fatou set, we argue by contradiction and assume that there exists a point $z\in B(q,r)\cap \Fa(f)$ with $h(z)\not\in \Fa(f)$. Then $h(z)\in \Ju(f)$. Since $B(q,r)\cap \Fa(f)$ is an open neighborhood of $z$, it follows again from  standard topological degree arguments that 
for  large enough $k\in \N$ there exists a point $z_k\in B(q,r)\cap \Fa(f)$ with 
$h_k(z_k)=h(z)\in \Ju(f)$. This is impossible by \eqref{eq: invarhk} and so indeed 
$h(B(q, r)\cap \Fa(f))\sub \Fa(f)$. We conclude that  
\begin{equation}\label{eq: invarh}
h^{-1}(\Ju(f))=\Ju(f)\cap B(q,r). 
\end{equation}

%
%

\medskip
{\em Step~IV.} We know by Theorem~\ref{thm:circgeom} that the peripheral circles of $\Ju(f)$ are uniformly relatively separated uniform quasicircles. 
According to Theorems~\ref{T:BoC1.2} and \ref{T:BoP5.1}, there exists a  quasisymmetric map $\beta$ on $\hC$ such that 
$S=\beta(\Ju(f))$ is a round Sierpi\'nski carpet. We may assume $S\sub \C$. 

We conjugate the map $f$ by $\beta$ to define a new map $\beta \circ f\circ \beta^{-1}$. By abuse of notation we call this new map also $f$. Note that this map and its iterates are in general not  rational anymore, but quasiregular maps on $\hC$. Similarly, we conjugate  $ g_\xi, h_k, h$ by $\beta$ to obtain new maps 
for which we use the same notation for the moment. 
If $V=\beta(B(q,r))$, then the new maps $h_k$ and $h$ are quasiconformal  
on $V$, and $h_k\to h$ uniformly on $V$.

\begin{lemma}\label{L:Eventuallyid}
There exist $N\in\N$  and an open set $W\sub V$ such that 
$S\cap  W$ is non-empty and connected and    
$h_k\equiv h$ on  $S\cap W$  for all  $k\geq N$. 
\end{lemma}

\begin{proof}
Since  $\Ju(f)$ is porous and $S$ is a quasisymmetric image of $\Ju(f)$, the set 
$S$ is also  porous (and in particular locally porous as defined in Section~\ref{s:Schottky}).

The maps $h_k$ and $h$ are quasiconformal on $V=\beta(B(q,r))$, and $h_k\to h$ uniformly on $V$ as $k\to \infty$. The relations \eqref{eq: invarh} and \eqref{eq: invarhk} translate to $h^{-1}(S)=S\cap V$ and $h_k^{-1}(S)=S\cap V$ for $k\in \N$.  So Lemma~\ref{L:Schmaps} implies that the maps $h\: S\cap V\ra S$ 
and $h_k\: S \cap V\ra S$ for $k\in \N$ are Schottky maps. Each of these restrictions is actually  a homeomorphism onto its image.

 There are only finitely many peripheral circles of $\Ju(f)$ that contain periodic points of our original rational map $f$; indeed, 
if $J$ is such a peripheral circle,  then $f^n(J)=J$ for some $n\in \N$ as  follows from Lemma~\ref{L:PerC}; but then $J$ bounds a periodic Fatou component of $f$ which leaves only  
finitely many possibilities for $J$. Since the periodic points of $f$ are dense in $\Ju(f)$, 
we conclude that we can find a periodic point of $f$ in $ \Ju(f)\cap B(q,r)$ that does not lie on a peripheral circle of $\Ju(f)$.

 Translated  to  the conjugated map $f$,  this yields  existence of  a point $a\in S\cap V$ that does not lie on a peripheral circle of the Sierpi\'nski carpet $S$  such that $f^n(a)=a$ for some $n\in \N$.  The invariance property of the Julia set gives  $f^{-n}(S)=S$, and so   Lemma~\ref{L:Schmaps} implies that  $f^n\:S\setminus\crit(f^n)\to S$ is a Schottky map. Note that $a\not\in \crit(f^n)$ as follows from the fact that for our original rational map $f$, none of its  periodic critical points lies in the Julia set.

Therefore,  our Schottky map $f^n\:S\setminus\crit(f^n)\to S$ has a derivative at the point $a\in S\setminus\crit(f^n)$ in the sense of \eqref{eq:confmean}.  
 If $(f^n)'(a)=1$, then Theorem~\ref{thm:dervuniq} implies that $f^n\equiv\id|_{S\setminus\bran(f^n)}$, and hence by continuity $f^n$ is the identity on $S$. 
This is clearly impossible, and therefore  $(f^n)'(a)\neq1$.

Since $a\in S\cap V$ does not lie on a peripheral circle of $S$, as in the proof of Lemma~\ref{L:Schmaps} we can find a small Jordan region  $W$ with $a\in W\sub V$ and $W\cap \crit(f^n)=\emptyset$ such that $\partial W\sub S$.  Then $S\cap W$ is  non-empty and connected. 

We now restrict our maps to $W$. Then $h_k\: S\cap W\ra S$ is a Schottky map and a homeomorphism onto its image for each $k\in \N$. The same is true for the map $h\: S \cap W\ra S$. Moreover $h_k\to h$ as $k\to \infty$ uniformly on $W\cap S$. Finally, the map
$u=f^n$ is defined on $S\cap W$ and gives a Schottky map $u\: S\cap W\ra S$ such that for $a\in S\cap W$ we have $u(a)=a$ and $u'(a)\ne 1$. 
So   we can apply Theorem~\ref{T:Me3T5.3} to conclude that there exists $N\in\N$ such that $h_k\equiv h$ on $S\cap W$ for all $k\geq N$. 
 \end{proof}

By the previous lemma we can fix $k\geq N$ so that $h_k=h_{k+1}$ on $S\cap W$.
   If we go back 
to the definition of the maps $h_k$ and use the consistency of inverse branches (which is  also true for the maps conjugated by $\beta$), then we conclude that 
$$ h_{k+1}=g_\xi^{m_{k+1}}\circ f^{-n_{k+1}}= h_k =g_\xi^{m_{k}}\circ f^{-n_{k}}=
g_\xi^{m_{k}}\circ f^{n_{k+1}-n_k}\circ f^{-n_{k+1}}$$
on the set $S \cap W$. 
Cancellation gives 
$$ g_\xi^{m_{k+1}}=
g_\xi^{m_{k}}\circ f^{n_{k+1}-n_k}$$
on $f^{-n_{k+1}}(S\cap W)\sub S$. 

The two maps on both sides of the last equation are quasiregular maps $\psi\: \hC\ra \hC$ with $\psi^{-1}(S)=S$. It follows from Lemma~\ref{L:Schmaps}  that they are Schottky maps $S \cap U\ra S$ 
if $U\sub \hC$ is an open set that does not contain any of the finitely many critical points of the maps; 
in particular,  $g_\xi^{m_{k+1}}$ and $g_\xi^{m_k}\circ f^{n_{k+1}-n_k}$ are Schottky maps 
$S \cap U\ra S$, where 
 $U=\hC \setminus (\crit(g_\xi^{m_{k+1}})\cup\bran(g_\xi^{m_k}\circ f^{n_{k+1}-n_k}))$. Since $U$ has a finite complement in $\hC$, the non-degenerate connected set  
 $f^{-n_{k+1}}(S\cap W)\sub S$ has an accumulation point in $S \cap U$.
Theorem~\ref{C:Me3C4.2}
yields 
$$ g_\xi^{m_{k+1}}=
g_\xi^{m_{k}}\circ f^{n_{k+1}-n_k}$$
on $S\cap U$, and hence on all of $S$ by continuity.

If we conjugate back by $\beta^{-1}$, this leads to the relation
$$ g_\xi^{m_{k+1}}=
g_\xi^{m_{k}}\circ f^{n_{k+1}-n_k}$$ on $\Ju(f)$ for the original maps. We conclude that there exist
 integers $m,m',n\in \N$ such that for the original maps we have
\begin{equation} \label{eq:main1}
 g^{m'}\circ \xi = g^m\circ \xi \circ f^n 
 \end{equation} 
on $\Ju(f)$.

\medskip 
{\em Step~V.} Equation~\eqref{eq:main1} gives us a crucial relation of $\xi$ to the dynamics of $f$ and $g$ on their Julia sets. We will  bring \eqref{eq:main1}  into a convenient form
by replacing our original maps with iterates. 
Since $\Ju(f)$ is backward invariant, counting preimages of  generic points in 
$\Ju(f)$ under iterates of $f$ and of points in $\Ju(g)$ under iterates of $g$ leads to the relation 
$$
 \deg(g)^{m'-m}=\deg(f)^n,
$$
and so $m'-m>0$. If we post-compose both sides in \eqref{eq:main1} by a suitable iterate of 
$g$, and then replace $f$ by $f^n$ and $g$ by $g^{m'-m}$,   
we arrive at a relation of the form 
\begin{equation} \label{eq:main3}
 g^{l+1}\circ \xi = g^l\circ \xi \circ f. 
 \end{equation}  
 on $\Ju(f)$ for some $l\in \N$. 
 
 Note that this equation implies that we have
 \begin{equation} \label{eq:main4}
 g^{n+k}\circ \xi = g^n\circ \xi \circ f^k \ \text{ for all $k,n\in \N$ with $n\ge l$}. 
 \end{equation}  

\medskip 
{\em Step~VI.} In this final step of the proof, 
we disregard the non-canonical  
extension to $\hC$  of  our original homeomorphism $\xi\: \Ju(f)\ra \Ju(g)$ chosen in the beginning. Our goal is to apply  
\eqref{eq:main4} to produce a natural    extension of $\xi$ 
 mapping  each Fatou component of $f$ conformally onto a Fatou component of $g$.  
Note that  if  $U$ is a Fatou component of $f$, then $\partial U$ is a peripheral circle of $\Ju(f)$.  
 Since $\xi$ sends each peripheral circle of $\Ju(f)$ to a peripheral circle of $\Ju(g)$, 
 the image  $\xi(\partial U)$ bounds a unique Fatou component $V$ of $g$.   This sets up a natural bijection between the Fatou components of our maps, and our goal is to conformally ``fill in  the holes". 
 
So let $\mathcal{C}_f$ and $\mathcal{C}_g$ be the sets of Fatou components of $f$ and $g$, respectively. By Lemma~\ref{lem:FatouDyn} we can choose a corresponding  
family $\{\psi_U: U\in \mathcal{C}_f\}$  of conformal maps. Since each Fatou component of $f$ is 
a Jordan region, we can consider $\psi_U$ as a conformal homeomorphism  from $\overline U$ onto 
$\overline \D$. Similarly, we obtain a family of conformal homeomorphisms  $\tilde \psi_V: \overline V\ra \overline \D$ for $V$ in $\mathcal{C}_g$. 
These Fatou components carry  distinguished basepoints 
$p_U=\psi_U^{-1}(0)\in U$ for $U\in  \mathcal{C}_f$ and  $\tilde p_V=\tilde \psi_V^{-1}(0)\in V$ for $V\in  \mathcal{C}_g$. 

 We will now first extend $\xi$ to the  periodic Fatou components of $f$, and then 
use the Lifting Lemma~\ref{lem:lifting} to get extensions to Fatou components of higher and higher level (as defined in the proof of Lemma~\ref{lem:FatouDyn}). In this argument it will be important to ensure that these extensions are basepoint-preserving. 
 
First let $U$ be a periodic Fatou component of  $f$. We denote by $k\in\N$  the period of $U$, and define $V$ to be the Fatou component of $g$ bounded by $\xi(\partial U)$,  and 
 $W=g^l(V)$, where $l\in \N$ is as in \eqref{eq:main4}. 
Then 
 \eqref{eq:main4}  implies that $\partial W$, and hence $W$ itself,  is invariant under $g^k$.  By 
 Lemma~\ref{lem:FatouDyn}  the basepoint-preserving homeomorphisms  $\psi_U\: (\overline U, p_U)\ra 
( \overline \D, 0)$ and $\tilde\psi_W\:(\overline W,\tilde p_W)\ra 
( \overline \D,0)$    conjugate $f^k$ and $g^k$, respectively, to power maps. 
Since $f$ and $g$ are postcritically-finite, the periodic Fatou components $U$ and $W$ are superattracting, and thus the degrees of these power maps are at least 2.

Again by Lemma~\ref{lem:FatouDyn} the map 
$\tilde \psi_W\circ g^l\circ\tilde \psi^{-1}_V$ is a power map. Since $U,V,W$ are Jordan regions, the maps $\psi_U, \tilde \psi_V, \tilde\psi_W$ give  homeomorphisms between the boundaries of the corresponding Fatou components and $\partial \D$. Since $\xi$ is an
orientation-preserving  homeomorphism of $\partial U$ onto $\partial V$, the map 
$\phi=\tilde \psi_V\circ \xi\circ\psi_U^{-1}$ 
 gives an orientation-preserving  homeomorphism  on $\partial \D$. Now \eqref{eq:main4} for $n=l$ implies that on $\partial \D$ we have 
\begin{align*} \label{eq:main5}
P_{d_3}\circ \phi&=\tilde  \psi_W\circ g^{k+l}\circ \tilde \psi_V^{-1} \circ\phi
= \tilde  \psi_W\circ g^{k+l}\circ \xi\circ \psi_U^{-1}\\
&=  \tilde  \psi_W\circ g^{l}\circ \xi\circ f^k\circ \psi^{-1}_U
= \tilde  \psi_W\circ g^{l}\circ \xi \circ \psi^{-1}_U \circ P_{d_1}\\
&=  \tilde  \psi_W\circ g^{l}\circ\tilde  \psi^{-1}_V\circ \phi\circ P_{d_1}
= P_{d_2}\circ \phi \circ P_{d_1}
\end{align*}
for some $d_1, d _2, d_3\in\N$ with $d_1\geq 2$.  Lemma~\ref{L:Rot} implies that $\phi$ extends to $\overline \D$ as a rotation around $0$, also denoted by $\phi$. In particular, $\phi(0)=0$, and so $\phi$ preserves the basepoint $0$ in $\overline \D$. If we  define $\xi=\tilde\psi_V^{-1}\circ \phi\circ\psi_U$  on $\overline U$, then $\xi$ is a conformal homeomorphism  
 of $(\overline U,p_U)$ onto  $(\overline V, \tilde p_V)$.

In this way,  we can  conformally extend $\xi$ to every periodic Fatou component of $f$ so that $\xi$ maps the  basepoint of a Fatou component to the basepoint of the image component. 
To get such an extension also for the other  Fatou components $V$ of $f$, we proceed inductively 
on the level of the Fatou component. So suppose that the level of $V$ is $\ge 1$ and that 
we have already found an extension 
for all Fatou components with a level lower than $V$. This applies to the Fatou component 
$U=f(V)$ of $f$, and so a  conformal extension  $(U, p_U)\ra ( U', \tilde p_{U'})$ of $\xi|_{\partial U}$ exists, where $U'$ is the Fatou component of $g$ bounded by $\xi(\partial U)$. 
 Let $V'$ be the Fatou component of $g$ bounded by $\xi(\partial V)$, and $W=g^{l+1}(V')$. 
 Then by using \eqref{eq:main4} on $\partial V$ we conclude that $g^l(U')=W$. 
 Define $\alpha=g^l\circ \xi|_{\overline U}\circ f|_{\overline V}$ and 
 $\beta=\xi|_{\partial V}$. Then the assumptions of Lemma~\ref{lem:lifting} are satisfied for $D=V$, $p_D=p_{V}$, and the iterate 
 $g^{l+1}\: V'\ra W$ of $g$. Indeed, $\alpha$ is continuous on $\overline V$ and holomorphic 
 on $V$, we have 
 $$\alpha^{-1}(p_W)=f^{-1}(  \xi|_{\overline U}^{-1}(\tilde p_{U'}))=f^{-1}(p_U)=\{p_V\},$$
 and 
 $$g^{l+1}\circ \beta= g^{l+1}\circ  \xi|_{\partial V}= g^{l}\circ \xi|_{\partial U}\circ f|_{\partial V}
 =\alpha.$$ Since $\beta$ is a homeomorphism, it follows  that there exists a conformal 
 homeomorphism  $\tilde \alpha$  of $(\overline V, p_V)$ onto $(\overline V', \tilde p_{V'})$ such that $\tilde \alpha|_{\partial V}=\beta=\xi|_{\partial V}$. In other words, 
 $\tilde \alpha$ gives the desired basepoint preserving conformal extension to the Fatou component $V$. 
 
  This argument shows  that $\xi$ has a (unique) conformal extension to each Fatou component of $f$. We know that the peripheral circles of $\Ju(f)$ are uniform quasicircles, and that $\Ju(f)$ has  measure zero. Lemma~\ref{lem:ext} now implies that $\xi$ 
extends to a M\"obius transformation on $\hC$, which completes  the proof. 


The techniques discussed also easily lead to a proof of  Corollary~\ref{C:FinGp}.

\begin{proof}[Proof of Corollary~\ref{C:FinGp}] Let $f$ be a postcritically-finite rational map whose Julia set $\Ju(f)$ is a Sierpi\'nski carpet. Let $G$ be the group of all 
M\"obius transformations $\xi$ on $\hC$  with $\xi(\Ju(f))=\Ju(f)$, and $H$ be the subgroup of all elements in $G$ that preserve orientation.  By Theorem~\ref{thm:main1} it is enough to prove  that $G$ is finite.   Since $G=H$ or $H$ has index $2$ in $G$, this is true if 
we can show that $H$ is finite.

Note that the group $H$ is {\em discrete}, i.e., there exists $\delta_0>0$ such that 
\begin{equation}\label{eq:discrete} \sup_{p\in \hC} \sigma(\xi(p), p)\ge \delta_0
\end{equation}
for all $\xi\in H$ with $\xi\ne \id_{\hC}$. Indeed, we choose $\delta_0>0$ so small that 
there are at least three distinct complementary components $D_1,D_2,D_3$ of $\Ju(f)$ that 
contain  disks of radius $\delta_0$. In order to show \eqref{eq:discrete}, suppose that $\xi\in H$ and  $\sigma(\xi(p), p)<\delta_0$ for all $p\in \hC$. Then $\xi(D_i)\cap D_i\ne \emptyset$, and so $\xi(\overline D_i)=\overline D_i$
for $i=1, 2,3$, because $\xi$ permutes the closures of  Fatou components of $f$.
This shows that $\xi$ is a conformal homeomorphism of the closed Jordan region $\overline D_i$ onto itself. Hence $\xi$ has a fixed point in $\overline D_i$. Since  $\Ju(f)$ is a Sierpi\'nski carpet, the closures $\overline D_1, \overline D_2, \overline D_3$ are pairwise disjoint, and we conclude that $\xi$ has at least three fixed points. Since $\xi$ is an orientation-preserving 
 M\"obius transformation, this 
implies that $\xi=\id_{\hC}$, and the discreteness of $H$ follows.

 We will now  analyze type of M\"obius transformations contained
in the group $H$   (for the relevant classification of M\"obius transformations up to conjugacy, see \cite[Section~4.3]{Be0}). 
So consider an arbitrary $\xi\in H$, $\xi\ne \text{id}_{\hC}$. 

Then $\xi$ cannot be loxodromic; indeed, otherwise $\xi$ has a repelling fixed point $p$ which necessarily has to lie in $\Ju(f)$. 
We now argue as in the proof of Theorem~\ref{thm:main2} and ``blow down" by iterates $\xi^{-n_k}$ near 
$p$ and ``blow up" by the conformal elevator  using iterates $f^{m_k}$ to obtain a sequence of conformal
maps of the form $h_k=f^{m_k}\circ \xi^{-n_k}$ that converge uniformly to a (non-constant) conformal limit function $h$ on a disk $B$ centered at a point in $q\in \Ju(f)$. Again this sequence stabilizes, and so $h_{k+1}=h_k$ for large $k$ on a connected non-degenerate subset of $B$, and hence on $\hC$ by the uniqueness theorem for analytic functions. This leads to a relation of the form 
$f^k\circ \xi^l=f^m$, where $k,l,m\in \N$. Comparing degrees we get $m=k$, and so 
we have  $f^{k}=f^{k}\circ \xi^{-ln}$ for all $n\in \N$. This is impossible, because
$f^k=f^{k}\circ \xi^{-ln}\to f^{k}(p)$ near $p$ as $n\to \infty$, while $f^k$ is non-constant. 

The M\"obius transformation $\xi$ cannot be parabolic either; otherwise, after conjugation we may assume that  
$\xi(z)=z+a$ with $a\in \C$, $a\ne 0$. Then necessarily $\infty\in \Ju(f)$. On the other hand, we know that 
the peripheral circles of $\Ju(f)$ are uniform quasicircles that occur on all locations and scales with respect to the chordal metric.  Translated to the Euclidean metric near $\infty$ this means that $\Ju(f)$ has complementary components $D$ with $\infty\not\in \partial D$ that contain   Euclidean disks of arbitrarily large radius, and in particular of radius $>|a|$; but then $\xi$ cannot move $D$ off itself, and so $\xi(D)=D$ for the translation $\xi$.
 This  is impossible.

Finally, $\xi$ can be elliptic; then after conjugation we have $\xi(z)=az$ with $a\in \C$ and $|a|=1$, $a\ne 1$.  Since $H$ is discrete, 
$a$ must be  a root of unity, and so  $\xi$
 is a torsion element of $H$.

We conclude that $H$ is a discrete group of M\"obius transformations such that 
each element  $\xi\in H$ with $\xi\ne \text{id}_{\hC}$ is a  torsion element. It is well-known that  such a  group $H$ is finite (one can derive this from  \cite[p.~70, Theorem 4.3.7]{Be0} in combination with the considerations in \cite[p.~84]{Be0}). 

\end{proof}


\begin{thebibliography}{BHa}


 \bibitem[AIM]{AIM} K.~Astala, T.~Iwaniec, G.M.~Martin, \emph{Elliptic partial differential equations and quasiconformal mappings in the plane}, Princeton Univ.\ Press, Princeton, NJ, 2009.  
 





\bibitem[Be1]{Be0} A. F. Beardon, {\em The geometry of discrete groups}, Springer, New York, 1983. 


\bibitem[Be2]{Be} A. F. Beardon, {\em Iteration of rational functions}, Springer, New York, 1991. 

\bibitem[B--S]{BDLSS} P. Blanchard, R. L. Devaney, D. M. Look, P.  Seal, Y. Shapiro, \emph{Sierpi\'nski-curve Julia sets and singular perturbations of complex polynomials,} Ergodic Theory Dynam. Systems 25 (2005), 1047--1055.


%


%



\bibitem[Bo]{Bo} M. Bonk, \emph{Uniformization of Sierpi\'nski carpets in the plane}, Invent. math. 186 (2011), 559--665.

\bibitem[BKM]{BKM} M. Bonk, B. Kleiner, S. Merenkov, \emph{Rigidity of Schottky sets}, 
Amer. J. Math. 131 (2009),  409--443. 

\bibitem[BM]{BM} M. Bonk, S. Merenkov, \emph{Quasisymmetric rigidity of square Sierpi\'nski carpets}, Ann. of Math. (2) 177 (2013), 591--643.





%



\bibitem[CG]{CG}  L. Carleson, Th.W. Gamelin, {\em Complex dynamics}, Springer, New York, 1993. 


\bibitem[DL]{DL} R. L. Devaney, D. M. Look, \emph{A criterion for Sierpi\'nski curve Julia sets,} Spring Topology and Dynamical Systems Conference. Topology Proc. 30 (2006),  163--179.



\bibitem[DH]{DH} A. Douady,  J. H. Hubbard, {\em A  proof of Thurston's  topological characterization  of rational  functions}, Acta Math.  171 (1993),  263--297. 














%





\bibitem[Ha]{Ha} A.~Hatcher, \emph{Algebraic topology}, Cambridge Univ. Press, Cambridge, 2002.

\bibitem[He]{He}J. Heinonen, \emph{Lectures on analysis on metric spaces,}
Springer, New York, 2001.













\bibitem[Le1]{Le} G. M. Levin, \emph{Symmetries on Julia sets} (Russian), Mat. Zametki 48 (1990),   72--79, 159; translation in Math. Notes 48 (1990),  1126--1131 (1991).

\bibitem[Le2]{Le1} G. M. Levin,  \emph{Letter to the editors: ``Symmetries on Julia sets"} (Russian), Mat. Zametki 69 (2001),  479--480; translation in Math. Notes 69 (2001), 432--433.

\bibitem[LP]{LP} G. Levin, F.  Przytycki, \emph{When do two rational functions have the same Julia set?} Proc. Amer. Math. Soc. 125 (1997), 2179--2190.

\bibitem[Ly]{Ly} M. Lyubich. {\em Typical behaviour of trajectories of a rational mapping of the sphere} (Russian),
  Dokl. Akad. Nauk SSSR 268 (1982), 29--32; translation in Soviet Math. Dokl. 27 (1983), 22--25. 



\bibitem[MSS]{MSS}  R. Ma\~n\'e, P. Sad and D. Sullivan.
\emph{On the dynamics of rational maps}, Ann. Scient. \'Ec. Norm. Sup.
16 (1983),  193--217.

\bibitem[Me1]{Me1} S.~Merenkov, \emph{A Sierpi\'nski carpet with the co-Hopfian property},  
Invent. math. 180 (2010), 361--388. 

\bibitem[Me2]{Me2} S.~Merenkov, \emph{Planar relative Schottky sets and quasisymmetric maps}, 
Proc. Lond. Math. Soc. (3) 104 (2012), 455--485. 

\bibitem[Me3]{Me3} S.~Merenkov,  \emph{Local rigidity of Schottky maps}, Proc. Amer. Math. Soc., to appear. Preprint, arXiv:1305.4683.

\bibitem[Me4]{Me4} S.~Merenkov,  \emph{Local rigidity for hyperbolic groups with Sierpi\'nski carpet boundaries}, Compos. Math. Soc., 
to appear. Preprint, arXiv:1307.1792.




\bibitem[Mi]{Mi2} J.~Milnor, {\em Dynamics in one complex variable}, 3rd ed., 
 Princeton Univ. Press, Princeton, NJ, 2006.



\bibitem[Mo]{Moo} R. L. Moore, \emph{Concerning upper semi-continuous collections of continua,} Trans.  Amer. Math. Soc. 27 (1925), 416--428.






\bibitem[Ri]{Ri} S. Rickman, {\em Quasiregular mappings},
Springer, Berlin, 1993.











\bibitem[St]{St} N. Steinmetz, \emph{Rational iteration}, de Gruyter,  
Berlin, 1993. 









\bibitem[V\"a]{Va} J.~V\"ais\"al\"a, \emph{Lectures on $n$-dimensional
quasiconformal mappings}, Lecture Notes in Mathematic 229,
Springer, Berlin-Heidelberg-New York, 1971.

%





\bibitem[Wh]{gW58} G. T. Whyburn, \emph{Topological characterization of the Sierpi\'nski curve,} Fund. Math. 45 (1958), 320--324.





\end{thebibliography}
\end{document}